\documentclass[12pt,a4paper]{amsart}
\usepackage[utf8]{inputenc}
\usepackage{amsmath}
\usepackage{amsfonts}
\usepackage{amssymb}
\usepackage{amsthm}
\usepackage{graphicx}
\usepackage{helvet}
\usepackage{amsfonts,amsmath,verbatim,amssymb,amscd,latexsym}
\usepackage{amsthm, calc, enumerate, amscd}
\usepackage[english]{babel}
\usepackage{hyperref}
\usepackage[margin=2.7cm]{geometry}

\newtheorem{theorem}{Theorem}

\newtheorem{lemma}{Lemma}

\begin{document}
	
	\author{Mohd Harun, Saurabh Kumar Singh}
	
	\title{Shifted convolution sum for $GL(3) \times GL(2)$ with weighted average}
	
	\address{ Saurabh Kumar Singh \newline {\em Department of Mathematics and Statistics, Indian Institute of Technology, Kanpur, India; \newline  Email: saurabs@iitk.ac.in
	} }

	\address{Mohd Harun \newline {\em Department of Mathematics and Statistics, Indian Institute of Technology, Kanpur, India; \newline  Email: harun@iitk.ac.in
	} }
\date{\today}
	
	\maketitle
	
	\begin{abstract}
		In this paper, we will prove the non-trivial bound for the weighted average version of shifted convolution sum for $GL(3) \times GL(2)$, i.e. for arbitrary small $\epsilon >0$ and $X^{1/4+\delta} \leq H \leq X$ with $\delta >0$, we prove
		\[
		\frac{1}{H}\sum_{h=1}^\infty \lambda_f(h) V\left( \frac{h}{H}\right)\sum_{n=1}^\infty  \lambda_{\pi}(1,n)  \lambda_g (n+h) W\left( \frac{n}{X} \right)\ll X^{1-\delta+\epsilon},
		\] 
		where $V,W$ are smooth and compactly supported functions, $\lambda_f(n), \lambda_g(n)$ and $\lambda_{\pi}(1,n)$ are the normalized $n$-th Fourier coefficients of holomorphic or Hecke-Maass cusp forms $f,g$ for $SL(2,\mathbb{Z})$, and Hecke-Maass cusp form $\pi$ for $SL(3,\mathbb{Z})$, respectively. 
	\end{abstract}

	\section{Introduction}

Let $X$ be a variable, and let $a(n)$ and $b(n)$ be arbitrary arithmetical functions. We define the corresponding shifted convolution sum as:
\[S(X) = \sum_{n \leq X} a(n) \, b(n + h),\]
where the shift $h$ is a non-zero integer. The problem of estimating these sums, denoted as $S(X)$, for various arithmetical functions $a(n)$ and $b(n)$ has a long history in analytic number theory. Nontrivial estimates of these sums have played significant roles in many central problems, such as subconvexity, quantum unique ergodicity, Chowla's conjecture, additive divisor problems, and numerous other intriguing questions (for further details, refer to the introduction section of \cite{r20}). Additional relevant information and developments can be found in references \cite{r2}, \cite{r1}, \cite{r6}, \cite{r12}, \cite{r15}, \cite{r16}, and \cite{r17}. For the sake of motivation, this section will focus solely on the literature directly related to our specific problem.
\vspace{0.1cm}

There is extensive literature concerning shifted convolution sums involving the $GL(2)$ Fourier coefficients. In 1995, Pitt examined a shifted convolution sum with one of the arithmetical functions being the ternary divisor function $d_3(n)$, which is the $n$-th coefficient of the Dirichlet series for $\zeta^3(s)$. Here $\zeta(s)$ denotes the classical Riemann zeta function. For the values $0 < r < X^{1/24}$ and $\epsilon > 0$, Pitt \cite{r18} demonstrated that:
\[
\sum_{n \leq X} d_3(n) \, \lambda_g(rn-1) \ll_{\zeta^3,f,\epsilon} X^{1-1/72 + \epsilon},
\]
where $\lambda_g(n)$ denotes the normalized $n$-th Fourier coefficient of a general holomorphic or Hecke-Maass cusp form for $SL(2, \mathbb{Z})$. This represented the first instance of a study of a shifted convolution sum involving the $GL(3)$ Fourier coefficients, albeit in a specific case. Following Pitt's result, a substantial body of literature has emerged regarding shifted convolution sums involving $ d_3(n)$. Recently, Munshi replaced $d_3(n)$ with the normalized $n$-th Fourier coefficients $\lambda_{\pi}(1,n)$ of a general Hecke-Maass cusp form $\pi$ for $SL(3, \mathbb{Z})$. Utilizing Jutila's variant of the circle method and the concept of factorizable moduli, Munshi proved \cite{r19} the following result (after a small correction in his Lemma $11$ of the published version):
\[
\sum_{n=1}^\infty \lambda_\pi(1,n) \, \lambda_g(n+h) \, V\left( \frac{n}{X}\right) \ll_{\pi,f, \epsilon} X^{1-1/26+ \epsilon},
\]
where $1 \leq |h| \leq X^{1+\epsilon}$ is an integer, and $V$ is a compactly supported smooth function. Later, Ping Xi \cite{r20} studied the same shifted convolution sum using a different approach for the involved character sum but still employing Jutila's variant of the circle method and obtained a better estimate of size $X^{1-1/22+\epsilon}$. Building upon Munshi's remark \cite{r19}, Sun \cite{r21} investigated the average version of the shifted convolution sum studied by Munshi. Employing the same variant of the circle method (Jutila's) and considering $r^{5/2}X^{1/4 + 7\delta/2} \leq H \leq X$ with $\delta > 0$, she established:
\[
\frac{1}{H} \sum_{h=1}^\infty V\left( \frac{h}{H}\right) \sum_{n=1}^\infty  \lambda(n) \,  \lambda_g(rn+h) \, W\left( \frac{n}{X} \right) \ll_{\pi,f,\epsilon} X^{1-\delta+ \epsilon},
\]
where $W$ represents another compactly supported smooth function, $r \geq 1$ is an integer, and $\lambda(n)$ is either $\lambda_{\pi}(1,n)$ or $d_3(n)$.

The objective of this paper is to extend Sun's result by considering the weighted average version of the shifted convolution sum, with the weight being the normalized $n$-th Fourier coefficients $\lambda_f(n)$ of a general holomorphic or Hecke-Maass cusp forms for $SL(2, \mathbb{Z})$, incorporated into the $h$ sum. More precisely, we are examining the following weighted average version of general $GL(3) \times GL(2)$ shifted convolution sums, denoted as $S(H, X)$:
\begin{equation} \label{A1}
	S(H,X) =  \frac{1}{H} \sum_{h=1}^\infty \lambda_f(h) \, V\left( \frac{h}{H}\right) \sum_{n=1}^\infty \lambda_{\pi}(1,n) \, \lambda_g (n+h) \,W\left( \frac{n}{X} \right).
\end{equation}
We will demonstrate that employing $\lambda_f(h)$ as the weight allows us to achieve additional savings in the character sum. By applying the Cauchy-Schwartz inequality and the estimates for the $GL(2)$ and $GL(3)$ Fourier coefficients from Rankin-Selberg theory, we obtain the following (trivial) estimate:
\begin{equation} \label{A2}
	S(H,X) \ll_{\pi, f, g, \epsilon} X^{1+\epsilon}. 
\end{equation}
In this article, by utilizing a different variant of the circle method developed by Duke, Friedlander, and Iwaniec, we will establish the following non-trivial estimate for $S(H, X)$ as the main result of this paper.
\begin{theorem}\label{main theorem}
		Let $f,g$ and $\pi$  be the Hecke-Maass cusp forms for the full modular groups $SL(2, \mathbb{Z})$  and $SL(3,\mathbb{Z})$, respectively. Let $\lambda_f(n),\lambda_g(n)$  and $\lambda_{\pi}(1,n)$ be the corresponding normalized $n$-th Fourier coefficients of the forms $f,g$  and $\pi$. Let $V_1$ and $V_2$ be the smooth functions and compactly supported on the interval $[1,2]$.\\
  
	(1)	For arbitrary small $\epsilon > 0$, and $X^{1/4 + \delta} \leq H \leq \sqrt{X} $ with $\delta > 0$, we have
		\begin{align*}
			\frac{1}{H}\,\sum_{h=1}^\infty\, \lambda_f(h)\, V_1\left( \frac{h}{H}\right)\,\sum_{n=1}^\infty \, \lambda_{\pi}(1,n)\,  \lambda_g (n+h)\, V_2\left( \frac{n}{X} \right)\, \ll \,\frac{X^{5/4 + \epsilon}}{H},
		\end{align*} where the implied constants are depending upon $f,g$ $\pi$ and $\epsilon$.\\

  (2)	For arbitrary small $\epsilon > 0$, and $\sqrt{X} \leq H \leq {X}$, we have
		\begin{align*}
			\frac{1}{H}\,\sum_{h=1}^\infty\, \lambda_f(h)\, V_1\left( \frac{h}{H}\right)\,\sum_{n=1}^\infty \, \lambda_{\pi}(1,n)\,  \lambda_g (n+h)\, V_2\left( \frac{n}{X} \right)\, \ll \,\frac{X^{1+\epsilon}}{H^{1/2}}.
		\end{align*} where the implied constants are depending upon $f,g$ $\pi$ and $\epsilon$.
	\end{theorem}
 
In this paper, we will also give a brief sketch for the proof of the following theorem where we are replacing the $GL(3)$ Fourier coefficient $\lambda_\pi(1,n)$ by the triple divisor function $d_3(n)$.
	
\begin{theorem}\label{TH2}
		Let $\lambda_f(n)$ and $\lambda_g(n)$ be as in the above theorem \eqref{main theorem}. Let $\lambda_{\pi}(1,n) = d_3(n)$, where $d_3$ is the triple divisor function.\\
  
  (1) For arbitrary small $\epsilon>0$ and $X^{1/4 + \delta} \leq H \leq \sqrt{X} $ with $\delta > 0$, we have
		\begin{align*}
			\frac{1}{H}\,\sum_{h=1}^\infty \,\lambda_f(h)\, V_1\left( \frac{h}{H}\right)\,\sum_{n=1}^\infty\, d_3(n)\,  \lambda_g (n+h)\, V_2\left( \frac{n}{X} \right)\, \ll\, \frac{X^{5/4 + \epsilon}}{H},
		\end{align*}where the implied constant will depend upon $f,g, \zeta^3$ and $\epsilon$.\\

  (2) For arbitrary small $\epsilon>0$ and $\sqrt{X} \leq H \leq X $, we have
		\begin{align*}
			\frac{1}{H}\,\sum_{h=1}^\infty \,\lambda_f(h)\, V_1\left( \frac{h}{H}\right)\,\sum_{n=1}^\infty\, d_3(n)\,  \lambda_g (n+h)\, V_2\left( \frac{n}{X} \right)\, \ll\, \frac{X^{1+\epsilon}}{H^{1/2}},
		\end{align*}where the implied constant will depend upon $f,g, \zeta^3$ and $\epsilon$.
	\end{theorem}

	\textbf{Notations:} For any complex number $z$, $e(z):= e^{2 \pi i z}$. By $A \asymp B$ we means that $ X^{-\epsilon} B \leq A \leq  X^{\epsilon}B$. By $A \sim B$ we mean that $B \leq A \leq 2B$. By the notation $X \ll Y$ we mean that there is a constant $c >0$ such that $|X| \leq c\, Y^{1+\epsilon} $. The implied constants during the calculations may depend on the cusp form $f,g,\pi$ and $\epsilon$. At various places, $\epsilon>0$ may be different.

\section{The Delta Method}\label{Delta}

In this paper, we are using a version of the delta method due to Duke, Friedlander, and Iwaniec. More specifically, we are using the expansion $(20.157)$ given in Chapter $20$ of \cite{r29}. The details for this section are taken from \cite[subsection 2.4]{r30}. Let $\delta: \mathbb{Z}\to \{0,1\}$ be defined by
	\[
	\delta(n, m)=
	\begin{cases}
		1 &\text{if}\,\,n=m \\
		0 &\text{otherwise}
	\end{cases}
	\]
	For $n,m\in\mathbb{Z}\cap [-2L,2L]$, we are using the following expression of the delta symbol to separate the oscillations
	\begin{equation}\label{s}
		\delta(n,m)=\frac{1}{Q}\,  \sum_{q \leq Q}\,\frac{1}{q}\,\,\, \sideset{}{^\star} \sum_{a\bmod q}e\left(\frac{(n-m) a}{q}\right)\int_{\mathbb{R}} \,\psi(q,x)\,e\left(\frac{(n-m) x}{q\,Q}\right) d x, 
	\end{equation} where $Q=2L^{1/2}$. Following properties of the function $\psi(q,u)$ are of our interest (see $(20.158)$ and $(20.159)$ of \cite{r29}, and \cite[Lemma 15]{r30})
	
	\begin{align}\label{delta}
		&\psi(q,x)=1+h(q,x),\,\,\,\,\text{with}\,\,\,\,h(q,x)=O\left(\frac{Q}{q }\left(\frac{q}{Q}+|x|\right)^A\right)\\
		&\psi(q,x)\ll |x|^{-A}\\
		&x^j \frac{\partial^j}{ \partial x^j} \psi(q, x) \ll \  \min \left\lbrace \frac{Q}{q}, \frac{1}{|x|} \right\rbrace \  \log Q
	\end{align}
	\vspace{0.2cm}
	
	for any $A>1$,  $j \geqslant 1$.  In particular, the second property implies that the effective range of integral in \eqref{s} is $[-L^{\epsilon}, L^{\epsilon}]$. It also follows that if $q \ll Q^{1- \epsilon}$ and $  x \ll Q^{- \epsilon} $, then  $ \psi(q, x)$ can
	be replaced by $1$ at a cost of a negligible error term. If $ q \gg Q^{1- \epsilon}$, then we get $ x^j \frac{\partial^j}{ \partial x^j} \psi(q, x) \ll Q^{\epsilon}$, for any $ j \geqslant 1$. If $ q \ll  Q^{1- \epsilon}$  and $  Q^{- \epsilon} \ll |x| \ll   Q^{ \epsilon}$, then $ x^j \frac{\partial^j}{ \partial x^j} \psi(q, x) \ll Q^{\epsilon}$, for any $ j \geqslant 1$. Hence in all cases, we can view $  \psi(q, x)$ as a nice weight function. 
	
\section{Sketch of the Proof (Theorem 1)}
In this section, we will give a brief sketch of the proof of our main Theorem \ref{main theorem}. Our main object of study is
\begin{align*}
 S(H,X) \,=\, \frac{1}{H}\sum_{h \sim H} \lambda_f(h) \, V_1\left(\frac{h}{H}\right) \sum_{n \sim X}^\infty  \lambda_{\pi}(1,n)  \lambda_{g}(n+h)\,V_2\left(\frac{n}{X}\right). 
\end{align*} After applying the delta method due to Duke, Friedlander, and Iwaniec, our object of study becomes,
\begin{align}\label{A105}
	S(H,X) = \frac{1}{HQ}&\sum_{q\sim Q}\frac{1}{q}\,\,\, \sideset{}{^\star} \sum_{a\bmod q}\int_{\mathbb{R}}\, \psi(q,u)
	\sum_{h \sim H} \lambda_f(h)  \,e\left(\frac{ah}{q}\right)\,V_1\left(\frac{h}{H}\right)\notag\\ 
	\times& \sum_{n \sim X}  \lambda_{\pi}(1,n)\, e\left(\frac{an}{q}\right)\,V_1\left(\frac{n}{X}\right)\,\sum_{m \sim Y}   \lambda_{g} (m)\,e\left(\frac{-am}{q}\right)\,V_1\left(\frac{m}{Y}\right)\,du.
\end{align}  For simplicity, we are taking the generic case i.e. $q \asymp Q = \sqrt{X}$, $Y \asymp X$, $1 \leq H \leq X$ and $V_1, V_2, V_3$ are compactly supported smooth functions supported on $[1,2]$. Note that the application of the delta symbol gives us a loss of size $X$. So we have to save $X$ and a little more in the expression for $S(H, X)$ given in equation \eqref{A105} to get our desired bound.\\

\textbf{Step 1: Application of $GL(2)$ Voronoi summation formula:}
We first apply the $GL(2)$ Voronoi summation formula given in Lemma \ref{voronoi} to sum over $m$ and $h$. We obtain that, their dual lengths are essentially supported on the size ${Q^2}/{Y}$ and ${Q^2}/{H}$. Hence, we obtained a saving $( =\sqrt{{\text{Initial Length}}/{\text{Dual Length}}})$ of size  ${Y}/{Q}$ and ${H}/{Q}$ in the two sums.\\ 

\textbf{Step 2: Application of GL(3) Voronoi summation formula:} Now we apply the $GL(3)$ Voronoi summation formula given in Lemma \ref{gl3voronoi} to sum over $n$. We obtain that the dual length is supported on the size ${Q^3}/{X}$. So, we got the saving of size ${X}/{Q^{3/2}}$ by this step in sum over $n$.\\
After the application of Voronoi summation formulae, we end up with the following object of study
\begin{align}\label{f1}
	S(H,X) \notag=& \frac{X^{2/3}Y^{3/4}}{Q^4\,H^{1/4}}\sum_{q\leq Q} \,\, \sum_{n_{2} \sim Q^3/X}  \frac{\lambda_{\pi}(n_{2},1)}{ n^{1/3}_{2}} \,  \\
	 &\times  \, \sum_{h\sim Q^2/H} \,\,\sum_{m\sim Q^2/Y} \frac{\lambda_f(h)}{h^{1/4}}  \frac{\lambda_g(m)}{m^{1/4}}\, \, \mathcal{C}(n_2, h, m, q)\, \mathcal{I}(...).
\end{align}
Here $\mathcal{I}(...)$ denotes the integral transform obtained on the application of Voronoi summation formulas. We will estimate $\mathcal{I}(...)$ trivially since it does not contain any oscillations. Also, $\mathcal{C}(n_2, h, m, q)$ denotes the character sum, which is given by
\begin{align*}
	\mathcal{C}(n_2, h, m, q) = \sideset{}{^\star} \sum_{a\, \bmod\, q} e\left( -\frac{\overline{a}h}{q}\right)e\left( \frac{\overline{a}m}{q}\right)  S\left( \bar{a},  \pm n_{2}; q\right)\,\rightsquigarrow\, q\,e\left( \frac{(\overline{h-m})n_2}{q}\right). 
\end{align*} Here $\rightsquigarrow$ means that the left-hand side essentially reduces to the right-hand side. We are saving $\sqrt{Q}$ in the $a$ sum. So far, our total saving over the trivial bound $X^2$ is given by
\begin{align*}
	\frac{Y}{Q}\times\frac{H}{Q}\times\frac{X}{Q^{3/2}}\times \sqrt{Q} \asymp \frac{X\,H}{Q}.
\end{align*} We need to save $Q/H$ and a little more to get our result.
We will get the additional savings in equation \eqref{f1}. Since the character sum reduces to an additive character in $n_2$ variable. So, when we apply the Cauchy- Schwartz inequality and then the Poisson summation formula to sum over $n_2$, we will save more than usual with the help of the congruence condition we obtained involving $n_2 \neq 0$ (See section \ref{nonzero} for details). \\

\textbf{Step 4: Cauchy- Schwartz inequality and Poisson summation formula:}
Next, we apply Cauchy-Schwartz inequality to sum over $n_2$, we obtained that our main sum $S(H, X)$ is dominated by the following expression
\begin{align*}
	\frac{X^{13/12}}{Q^2\,H^{1/4}}\,\left(\sum_{n_{2} \sim Q^3/X} \left |\,\sum_{q \sim Q}\,\sum_{h\sim Q^2/H}\,\frac{\lambda_f(h)}{h^{1/4}} \, \sum_{m\sim Q^2/Y} \frac{\lambda_g(m)}{m^{1/4}}\, \,e\left( \frac{(\overline{h-m})n_2}{q}\right)\, \,\right|^2\right)^{1/2}.
\end{align*} Opening the absolute valued square, we apply the Poisson summation formula to sum over $n_2$. In the zero frequency $n_2 =0$, we save the whole length of the diagonal i.e. $\frac{Q^3}{H}$. So we are saving more than we need in this case. In the non-zero frequency $n_2 \neq 0$, we are saving an extra $q$ because of the additive character involved. So our total savings in this case is of size $Q$, which is good enough as long as $H > X^{1/4}$.

\section{Preliminaries}	
In this section, we shall briefly recall some basic facts about $SL(2,\mathbb{Z})$ and $SL(3,\mathbb{Z})$ automorphic forms. We required very minimal details about these concepts. In fact, the Voronoi summation formulas and some standard estimates on Fourier coefficients of these $SL(2,\mathbb{Z})$ and $SL(3,\mathbb{Z})$ automorphic forms are all that we are using in this paper.

\subsection{Back ground on the Hecke-Maass cusp forms for $\bf SL(2,\mathbb{Z})$}
Let $f$ be a primitive holomorphic Hecke eigen form of integral weight $k$ for the full modular group $ SL(2, \mathbb{Z})$. For $z= x+ iy, y>0$, the normalized Fourier expansion of $f$  at the cusp $\infty$ is given by 
$$ f(z)= \sum_{n=1}^\infty \lambda_f (n) n^{(k-1)/2} e(nz) \hspace{1cm} ( \lambda_f(1) =1),$$ 
where $e(z) = e^{2\pi i z}$.  
Analogously, let $f(z)$ be a primitive Hecke-Maass cusp form for the group $SL(2, \mathbb{Z})$ with Laplacian eigenvalue $ \frac{1}{4} + \nu^2$.   The normalized Fourier expansion of $f$  at the cusp $\infty$ is given by 
\[   
2 \sqrt{y} \sum_{n \neq 0 } \lambda_f(n) K_{i\nu}  (2\pi |n| y ) e(n x), 
\] where $ K_{i\nu}$ denotes the $K$-Bessel function and $ \lambda_f(1) =1$. It follows from the Rankin-Selberg theory that the Fourier coefficients $\lambda_f (n)  $\textquotesingle s are bounded on average, namely:
\begin{equation}\label{s16}
	\sum_{n\leq X} |\lambda_f (n)|^2 = C_f X + O \left( x^{3/5} \right), 
\end{equation} for some constant $C_f > 0$. Ramanujan-Petersson conjecture predicts that $\lambda_f(n) \ll n^\epsilon$. This has been proved by Deligne in the case of holomorphic cusp forms, where he proves that $\lambda_f (n) \ll d(n)$. In the case of Maass cusp form the best-known result is  $\lambda_f (n) \ll n^{7/64 + \epsilon}$, proved by Kim and Sarnak( see \cite{r25}).

\vspace{0.5cm}

We shall use the following Voronoi type summation formula, which was first proved by Meurman \cite{r26}. 
\begin{lemma} \label{voronoi}
	{\bf Voronoi summation formula}:
	Let $u$ be a smooth and compactly supported function on the interval $(0, \infty)$ and $\lambda(n)$ be the nth Fourier coefficients of a Hecke-Maass forms for $SL(2, \mathbb{Z})$. We have
	\begin{equation} \label{varequation}
		\sum_{n=1}^\infty \lambda (n)\, e_q(an)\, u(n) = \frac{1}{q}\, \sum_{\pm} \sum_{n=1}^\infty\, \lambda( n) \,e_q(\pm\, \overline{a}n)\, U^{\pm} \left( \frac{n}{q^2}\right),
	\end{equation}
	where $ a \overline{a} \equiv 1 (\textrm{mod} \  q)$, and 
	\begin{align*}
		&U^{-} (y)= \frac{- \pi}{\sin( \pi i\nu)} \int_0^\infty u(x) \left\lbrace  Y_{2i\nu } + Y_{-2i\nu }\right\rbrace \left( 4\pi \sqrt{xy}\right) dx, \\
		&U^{+} (y)= 4\cosh( \pi \nu) \int_0^\infty u(x)  K_{2i\nu }  \left( 4\pi \sqrt{xy}\right) dx,
	\end{align*} where $Y_{2i\nu } $ and $ K_{2i\nu }$ are Bessel's functions of first and second kind and $e_q(x)= e^{\frac{2 \pi i x}{q}}$.  
\end{lemma}
\begin{proof}
	See \cite{r26} or appendix in \cite{r12}.
\end{proof}

\textbf{Remark 1}. If the function $u(x)$ is supported on $[X, 2X ]$ and satisfies $x^j\,u^{(j)}(x) \ll 1$. Using the properties of the Bessel's functions given in Lemma \ref{A4} and then repeated integration by parts, we can easily deduce that the integral $U^{\pm}$ is negligibly small if $n \gg q^2(qX)^{\epsilon}/X$. Hence the sum on the right-hand side of equation \eqref{varequation} is essentially supported on
$n \ll q^2(qX)^{\epsilon}/X$. For smaller values of $n$, we will use the trivial bound $U^{\pm}(n/q^2) \ll X$.

\subsection{Back ground on the Maass forms for $\bf SL(3,\mathbb{Z})$} \label{gl3 maass form} 
Let $\pi$ be a Maass form of type $(\nu_{1}, \nu_{2})$ for $SL(3, \mathbb{Z})$. We introduce the Langlands parameters $({\bf \alpha}_{1},{\bf \alpha}_{2},{\bf \alpha}_{3})$ of $\pi$, which are given by
$${\bf \alpha}_{1} = - \nu_{1} - 2 \nu_{2}+1, \, {\bf \alpha}_{2} = - \nu_{1}+ \nu_{2},  \, {\bf \alpha}_{3} = 2 \nu_{1}+ \nu_{2}-1.$$
By the work of Jacquet, Piatetskii-Shapiro, and Shalika, we have the Fourier Whittaker expansion for the Maass form $\pi(z)$ as follows
\begin{equation} \label{four-whi-exp}
	\pi_1(z) = \sum_{\gamma \in U_{2}\left(\mathbb{Z}\right) \backslash  SL(2,\mathbb{Z})} \sum_{m_{1}=1}^{\infty} \sum_{m_{2} \neq 0} \frac{\lambda_{\pi}(m_{1},m_{2})}{m_{1} |m_{2}|} \, W_{J}\left(M \left( {\begin{array}{cc}
			\gamma &  \\
			& 1 \\
	\end{array} } \right)
	z, \nu, \psi_{1,\frac{m_2}{|m_2|}}\right),
\end{equation}
where $z$ lies inside the corresponding fundamental domain for $SL(3, \mathbb{Z})$, $U_{2}(\mathbb{Z})$ is the group of upper triangular matrices with integer entries and ones on the diagonal,
$W_{J}\left(z,\nu,\psi_{1,1}\right)$ is the Jacquet-Whittaker function, and $M=\text{diag} \left(m_{1}|m_{2}|,m_{1},1\right)$ (cf. Goldfeld \cite{r24}). We consider $\pi$ to be the eigenfunction of all the Hecke operators with Fourier coefficients $\lambda_{\pi}(m_1, m_2)$, normalized so that $\lambda_{\pi}(1,1) = 1$

With the aid of the above terminology, we state the $GL(3)$-Voronoi summation formula in the following proposition.
\begin{lemma} \label{gl3voronoi}
	Let $\psi (x)$ be a compactly supported smooth function on $(0,\infty)$. Let $\lambda_{\pi}(m,n)$ be the $(m,n)$-th Fourier coefficient of a Maass form $\pi(z)$ for $SL(3,\mathbb{Z})$. Then we have
	\begin{align} \label{GL3-Voro}
		& \sum_{n=1}^{\infty} \lambda_{\pi}(m,n) e\left(\frac{an}{q}\right) \psi(n) \\
		\nonumber & =q  \sum_{\pm} \sum_{n_{1}|qm} \sum_{n_{2}=1}^{\infty}  \frac{\lambda_{\pi}(n_{2},n_{1})}{n_{1} n_{2}} S\left(m \bar{a}, \pm n_{2}; mq/n_{1}\right) \, G_{\pm} \left(\frac{n_{1}^2 n_{2}}{q^3 m}\right),
	\end{align} 
	where $G_{\pm}(x)$ is the integral transform, $(a,q)=1, \bar{a}$ is the multiplicative inverse modulo $q$ and $$S(a,b;q) = \sideset{}{^\star}{\sum}_{x \,\rm mod \, q} e\left(\frac{ax+b\bar{x}}{q}\right) $$
	is the Kloostermann sum.
\end{lemma}
\begin{proof}
See \cite{r31}.
\end{proof}

The following lemma gives an asymptotic expansion for the integral transform $G_{\pm}(x)$.
\begin{lemma} \label{GL3oscilation}
	Let $G_{\pm}(x)$ be as above,  and  $g(x) \in C_c^{\infty}(X,2X)$. Then for any fixed integer $K \geq 1$ and $xX \gg 1$, we have
	\begin{equation*}
		G_{\pm}(x)=  x \int_{0}^{\infty} g(y) \sum_{j=1}^{K} \frac{c_{j}({\pm}) e\left(3 (xy)^{1/3} \right) + d_{j}({\pm}) e\left(-3 (xy)^{1/3} \right)}{\left( xy\right)^{j/3}} \, \mathrm{d} y + O \left((xX)^{\frac{-K+5}{3}}\right),
	\end{equation*}
	where $c_{j}(\pm)$ and $d_{j}(\pm)$ are some  absolute constants depending on $\alpha_{i}$,  $i=1,\, 2,\, 3$.  
\end{lemma}
\begin{proof}
	See  \cite{r34}.
\end{proof}

The following lemma gives the Ramanujan bound for $\lambda_{\pi}(m,n)$ on average.	
\begin{lemma} \label{ramanubound}
	We have 
	$$\mathop{\sum \sum}_{n_{1}^{2} n_{2} \leq X} \vert \lambda_{\pi}(n_{1},n_{2})\vert ^{2} \ll \, X^{1+\epsilon}.$$
\end{lemma}
\begin{proof}
	Proof can be found in the book by Goldfeld \cite{r24}.
\end{proof}

\begin{lemma} \label{poisson}
	{\bf Poisson summation formula}: Let $F:\mathbb{R } \rightarrow \mathbb{R}$ is any Schwarz class function.  The Fourier transform 
	of $F$ is defined  as 
	\[
	\widehat{F}(y) = \int_{ \mathbb{R}} F( x) e(- x   y) dx,
	\] where $dx$ is the usual Lebesgue measure on $ \mathbb{R } $. 
	We have the Poisson summation formula defined as
	\begin{equation*}
		\sum_{ n \in \mathbb{Z}  }F(n) = \sum_{m \in \mathbb{Z} } \widehat{F}(m). 
	\end{equation*}
	Also, when $W(x)$ is any smooth and compactly supported function on $\mathbb{R}$, we have:
	\begin{align*}
		\sum_{n \in \mathbb{Z}  }e\left( \frac{an}{q}\right) W\left( \frac{n}{X}\right) = \frac{X}{q} \sum_{ m \in \mathbb{Z}  } \sum_{\alpha (\textrm{mod} \ q )}  e\left(\frac{( a + m) \alpha}{q} \right) \widehat{W} \left( \frac{mX}{q} \right). 
	\end{align*} 
\end{lemma}
\begin{proof}
	See  \cite[page 69]{r29}.
\end{proof}

\textbf{Remark 2:} If $W(x)$ satisfies $x^j\, W^{(j)}(x) \ll 1$, then it can be easily shown using the integrating by parts that the dual sum is essentially supported on $m \ll q(qX)^{\epsilon}/X$. The contribution coming from $m \gg q(qX)^{\epsilon}/X$ is negligibly small.

\begin{lemma}\label{A4}
	Let $ J_k (y)$ be the Bessel function of the first kind of weight k and $Y_{\pm 2 i\nu } (y) $ be the Bessel function of the second kind. Let $ K_\nu $ denote the modified Bessel function of the second kind. We have
	\begin{align*} \label{bessel}
		J_{k-1} (y)&, \ Y_{\pm2 i\nu}(y)  = y^{-1/2}(e^{iy} \,V(y)  + e^{-iy}\,  \overline{V}(y)),
	\end{align*}
where $ V(y)$ is a compactly supported smooth function and also satisfies 
	$$ y^j V^{(j)} (y)  \ll_{j} 1/\sqrt{(1+y)}.$$ 	
\end{lemma}
\begin{proof}
	See appendix in \cite{r12}.
\end{proof}
	
	\section{Proof of the Theorem \ref{main theorem}}
	We want to prove the cancellation in the following sum
	\begin{equation}
		S(H,X) = \, \frac{1}{H}\sum_{h=1}^\infty \lambda_f(h) V_1\left( \frac{h}{H}\right)\sum_{n=1}^\infty  \lambda_{\pi}(1,n)  \lambda_g (n+h) V_2\left( \frac{n}{X} \right),
	\end{equation} where $V_1, V_2$ are smooth and compactly supported functions, supported on the interval $[1,2]$. Let $Y = X+h$, we note that $ Y \asymp X$. Our first step is to introduce the delta symbol to separate the oscillations i.e. we can write
 
\begin{equation*}
		S(H,X) = \, \frac{1}{H}\sum_{h=1}^\infty \lambda_f(h) V_1\left( \frac{h}{H}\right)\sum_{n=1}^\infty  \lambda_{\pi}(1,n)\,V_2\left( \frac{n}{X} \right) \,\sum_{m=1}^\infty  \lambda_g (m) V_3\left( \frac{m}{Y} \right)\,\delta(n+h,m),
\end{equation*}
where $V_3$ is another smooth function that is supported on the interval $[1,2]$. By using the expression for $\delta$ symbol given in equation \eqref{s}, we can rewrite our main sum as,
	\begin{align*}
		S(H,X)= &\,\frac{1}{HQ}\sum_{q\leq Q}\frac{1}{q}\,\,\, \sideset{}{^\star} \sum_{a\bmod q}\,\int_{\mathbb{R}} W(u)\psi(q,u)\,\\
  & \times\,\sum_{h=1}^\infty \lambda_f(h)  \,e\left(\frac{ah}{q}\right)\,e\left(\frac{hu}{qQ}\right) V_1\left( \frac{h}{H}\right)\\ 
		&\times\, \sum_{n=1}^\infty  \lambda_{\pi}(1,n)\, e\left(\frac{an}{q}\right)\,e\left(\frac{nu}{qQ}\right) V_2\left( \frac{n}{X}\right)\\
		&\times\, \sum_{m=1}^\infty   \lambda_g (m)\,e\left(\frac{-am}{q}\right)\,e\left(\frac{-mu}{qQ}\right)  V_3\left( \frac{m}{Y}\right)du,
	\end{align*}
where $Q \asymp \sqrt{X}$,  $W$ is a smooth function supported on the interval $[-2X^{\epsilon}, 2X^{\epsilon}]$ with $W(u)=1 \,\,\forall u \in [-X^{\epsilon}, X^{\epsilon}] $ and $W^{(j)} \leq 1$. In the next subsection, we will apply the Voronoi summation formulae to sum over $h, n$, and $m$.   

\subsection{Applying the Voronoi summation formulas}
We are considering the case when all three $f,g$, and $\pi$ are Hecke-Maass cusp forms. There are almost similar even simpler calculations for the case of holomorphic cusp forms. We can further rewrite the sum $S(H, X)$ as
	\begin{align}\label{A17}
		S(H,X) &=\frac{1}{HQ}\sum_{q\leq Q}\frac{1}{q}\,\,\, \sideset{}{^\star} \sum_{a\bmod q}\int_{\mathbb{R}} W(u)\,\psi(q,u)\,T_1(...)\,T_2(...)\,T_3(...), \hspace{1cm} \text{where}
	\end{align} 
 
	\begin{align}\label{A93}
		 T_1(...) = \sum_{h=1}^\infty \lambda_f(h)  \,e\left(\frac{ah}{q}\right)\, v_1(h),
   \end{align}
   \begin{align}\label{s1}
		T_2(...) = \sum_{n=1}^\infty  \lambda_{\pi}(1,n)\, e\left(\frac{an}{q}\right)\,v_2(n),
  \end{align}
  \begin{align}\label{s2}
		T_3(...) =  \sum_{m=1}^\infty   \lambda_g (m)\,e\left(\frac{-am}{q}\right)\,v_3(m).
	\end{align} Here we are taking
 \begin{align*}
		v_1(y) = V_1\left( \frac{y}{H}\right) e\left(\frac{yu}{qQ}\right),\,\,\,\,\,\,\,\,
		v_2(y) = V_2\left( \frac{y}{X}\right) e\left(\frac{yu}{qQ}\right),\,\,\,\,\,\,\,\,
		v_3(y) = V_3\left( \frac{y}{Y}\right)e\left(\frac{-yu}{qQ}\right).
	\end{align*} 
	
	First, we will apply the $GL(2)$ Voronoi summation formula to sum over $h$, and $m$. Then we will apply the $GL(3)$ Voronoi summation formula to sum over $n$. We will capture the details in the next three lemmas. 
 
	\begin{lemma}\label{lemma8}
		Let $T_1(...)$ be as given in equation \eqref{A93}. We have
		\begin{align*}
			T_1(...) = \frac{{H^{3/4}}}{\sqrt q}\, \sum_{h\ll q^2/H }^\infty \frac{\lambda_f(h)}{h^{1/4}}\, e\left(-\frac{\overline{a}h}{q}\right)\, \mathcal{I}_1(h,u,q) + O(X^{-A})
		\end{align*} where $A$ is large positive real number and
		\begin{align*}
			&\mathcal{I}_1(h,u,q) =  \int_0^\infty  V_1(x)\,  e\left(\frac{Hxu}{qQ}\pm \frac{2 \sqrt{Hhx}}{q} \right)dx.
		\end{align*}
	\end{lemma}
	\begin{proof}
		On applying the $GL(2)$ Voronoi summation formula given in Lemma \ref{voronoi} to sum $T_1(...)$, we get
		\begin{align}\label{A5}
			T_1(...) = \frac{1}{q} \sum_{\pm} \sum_{h=1}^\infty \lambda_f(h) e\left(- \frac{\overline{a}h}{q}\right) \,\mathcal{V}^{\pm}_1 \left( \frac{h}{q^2}\right),
		\end{align}  
		where $\mathcal{V}^{\pm}_1$ represents the integral transforms defined in Lemma \eqref{voronoi}, i.e.
  \begin{align}\label{A7}
			\mathcal{V}^{+}_1 \left( \frac{h}{q^2}\right)= 4\cosh( \pi \nu_1) \int_0^\infty v_1(x)  K_{2i\nu_1 }  \left(\frac{4\pi \sqrt{hx}}{q}\right) dx, \hspace{1cm} \text{and}
		\end{align}
		\begin{align}\label{A6}
			\mathcal{V}^{-}_1 \left( \frac{h}{q^2}\right)= \frac{- \pi}{\cosh( \pi \nu_1)} \int_0^\infty v_1(x) \left\lbrace  Y_{2i\nu_1 } + Y_{-2i\nu_1 }\right\rbrace \left(\frac{4\pi \sqrt{hx}}{q}\right) dx.
		\end{align}
		 where $Y_{2i\nu_1 } $ and $ K_{2i\nu_1 }$ are the Bessel's functions of first and second kind. We first deal with $\mathcal{V}^{+}_1 \left( \frac{h}{q^2}\right)$. Putting the value of $v_1$, and changing the variable $x \longrightarrow Hx$,  we get 
		\begin{align*}
			\mathcal{V}^{+}_1 \left( \frac{h}{q^2}\right)= 4\cosh( \pi \nu_1)\,H\, \int_0^\infty  V_1\left( x\right) e\left(\frac{Hxu}{qQ}\right)  K_{2i\nu_1 }\left(\frac{4\pi \sqrt{Hhx}}{q}\right) dx.
		\end{align*} Using the approximation of the Bessel function from Lemma \ref{A4}, we see that $\mathcal{V}^{+}_1 \left( \frac{h}{q^2}\right)$ will be essentially the sum of two integrals of the form		
           \begin{align}\label{s4}
			&\,\frac{ H^{3/4}\,\sqrt{q}\,}{h^{1/4}} \int_0^\infty V_1(x) \, e\left(\frac{Hxu}{qQ}\pm \frac{2 \sqrt{Hhx}}{q} \right)dx \,=\, \frac{H^{3/4}\sqrt{q}}{h^{1/4}}\,\mathcal{I}_1(h,u,q),
		\end{align}
  where 
  \[\mathcal{I}_1(h,u,q) =  \int_0^\infty  V_1(x)\,  e\left(\frac{Hxu}{qQ}\pm \frac{2 \sqrt{Hhx}}{q} \right)dx.\] 
  Notice the slight abuse of notation, the weight function $V_1(x)$ is different from the one in the previous expression. Changing the variable $x \longrightarrow x^2$ and using integration by parts $j$-times, we get
  \begin{align*}
   \mathcal{I}_1(h,u,q) \ll\, \left(1 + \frac{Hu}{qQ}\right)^j\, \left(\frac{q}{\sqrt{Hh}}\right)^{j}.   
  \end{align*}
Hence, $\mathcal{I}_1(h,u,q)$ will be negligibly small unless $h \ll q^2/H$.

  Similarly, in the case of $\mathcal{V}^{-}_1 \left( \frac{h}{q^2}\right)$, by following the same steps, we will essentially get a sum of four integrals of the form as in equation \eqref{s4} and giving the same restriction $h \ll q^2/H$ on the dual length. Finally, by putting all these observations into equation \eqref{A5}, we get 
		\begin{align}\label{A9}
			T_1(...) = \,\frac{{H^{3/4}}}{\sqrt q}\, \sum_{h\ll q^2/H }^\infty \frac{\lambda_f(h)}{h^{1/4}}\, e\left(-\frac{\overline{a}h}{q}\right)\, \mathcal{I}_1(h,u,q) + O(X^{-A}),
		\end{align} where $A$ is large positive real number.
		\vspace{0.4cm}
	\end{proof}
	
	\begin{lemma}\label{lemma9}
		Let $T_3(...)$ be as given in \eqref{A93}. We have 
		\begin{align*}
			T_3(...) = \, \frac{{Y^{3/4}}}{\sqrt q} \, \sum_{m\ll q^2/Y }^\infty \frac{\lambda_g(m)}{m^{1/4}} e\left( \frac{\overline{a}m}{q}\right) \mathcal{I}_3(m,u,q) + O(X^{-A}),
		\end{align*} where $A$ is large positive real number and,
		\begin{align*}
			\mathcal{I}_3(m,u,q)\,=\,\int_0^\infty  V_3(y)\, e\left(-\frac{Yuy}{qQ}\pm \frac{2 \sqrt{Ymy}}{q} \right)dy.
   \end{align*}
	\end{lemma}
	\begin{proof}
		Following the same steps as in the above lemma \eqref{lemma8}, we will get our desired result.
	\end{proof}
		
  \begin{lemma}\label{lemma10}
		Let $T_2(...)$ be as given in \eqref{A93}. We have 
\begin{align*}
		T_2(...) = \frac{X^{2/3}}{q} \sum_{\pm} \sum_{n_{1}|q} \sum_{n^2_1n_{2}\ll N}  \frac{\lambda_{\pi}(n_{2},n_{1})}{n^{-1/3}_{1} n^{1/3}_{2}} S\left( \bar{a}, \pm n_{2}; q/n_{1}\right) \mathcal{I}_2(n^2_1n_2,u,q) + O(X^{-A}),
	\end{align*} where $ N= \text{max}\left\{{q^3}/{X} + X^{1/2}u^3\right\}$, $A$ is any positive real number and,
	\begin{align*}
		\mathcal{I}_2(n^2_1n_2,u,q) = \,\int_{0}^\infty V_2(z)\, e\left(\frac{Xuz}{qQ} \pm  \frac{3(Xz n_1^2 n_2)^{1/3}}{q} \right)dz .
	\end{align*}
 \end{lemma}
	\begin{proof}
On applying the $GL(3)$-Voronoi summation formula from Lemma \ref{gl3voronoi} to the sum $T_2(...)$, we get
		\begin{align} \label{A10}
			T_2(...) = q \sum_{\pm} \sum_{n_{1}|q} \sum_{n_{2}=1}^{\infty}  \frac{\lambda_{\pi}(n_{2},n_{1})}{n_{1} n_{2}} S\left( \bar{a}, \pm n_{2}; q/n_{1}\right) \, \mathcal{V}^{\pm}_2 \left(\frac{n_{1}^2 n_{2}}{q^3 }\right).
		\end{align} 
By using Lemma \ref{GL3oscilation} for the asymptotic expansion of the integral transform $\mathcal{V}^{\pm}_2 \left(\frac{n_{1}^2 n_{2}}{q^3 }\right)$ and simplifying it. We get that, up to a negligible error term, we can rewrite our integral transform as 
  \begin{align*}
		\mathcal{V}^{\pm}_2 \left(\frac{n_{1}^2 n_{2}}{q^3 }\right) = \frac{X^{2/3}}{q^2} (n_{1}^2 n_{2})^{2/3}\, \mathcal{I}_2(n^2_1n_2,u,q)+ O(X^{-A}),
	\end{align*} 
where $A$ is some large positive real number and
\begin{align*}
\mathcal{I}_2(n^2_1n_2,u,q) = \, \int_{0}^\infty V_2(z)\, e\left(\frac{Xuz}{qQ} \pm \frac{3(Xz n_1^2 n_2)^{1/3}}{q} \right)dz.  
\end{align*}
Here $V_2$ is a new weight function. One can look at \cite[section 3.2]{r35} for more details. Now, using integration by parts $j$-times, we get 
\begin{align*}
		\mathcal{I}_2(n^2_1n_2,u,q) \ll \left(1 + \frac{Xu}{qQ}\right)^j\, \left(\frac{q}{(Xn^2_1n_2)^{1/3}}\right)^j.
		\end{align*}
Hence, the integral $\mathcal{I}_2(n^2_1n_2,u,q)$ will be negligibly small if 
\begin{align*}
n^2_1n_2 \gg \,\text{max}\left\{\frac{q^3}{X} + X^{1/2}u^3\right\} =: N.
		\end{align*}
In this way, we say that $\mathcal{V}^{\pm}_2 \left(\frac{n_{1}^2 n_{2}}{q^3}\right)$ is negligibly small if \,$n_1^2n_2 \gg {N}$. From equation \eqref{A10}, we get 
\begin{align}\label{A16}
			T_2(...)= \frac{X^{2/3}}{q} \sum_{\pm} \sum_{n_{1}|q} \sum_{n^2_1n_{2}\ll N}  \frac{\lambda_{\pi}(n_{2},n_{1})}{n^{-1/3}_{1} n^{1/3}_{2}} S\left( \bar{a}, \pm n_{2}; q/n_{1}\right) \mathcal{I}_2(n^2_1n_2,u,q) +O(X^{-A}).
		\end{align}
This is our required result.  
\end{proof}
 In the next lemma, we will summarise all the calculations done so far in this subsection. 
 \begin{lemma}\label{lemma11}
		Let $S(H,X)$ be as given in equation \eqref{A17}, we have
		\begin{align}
			S(H,X) =\, \notag &\frac{X^{2/3}Y^{3/4}}{QH^{1/4}}\,\sum_{q\leq Q}\frac{1}{q^3}\,\, \sum_{\pm} \sum_{n_{1}|q}\, \sum_{n_{2} \ll N_0} \frac{\lambda_{\pi}(n_{2},n_{1})}{n^{-1/3}_{1} n^{1/3}_{2}}\\
			&\times \,\sum_{h \ll H_0} \frac{\lambda_f(h)}{h^{1/4}}\, \sum_{m \ll M_0} \frac{\lambda_g(m)}{m^{1/4}}\,\mathcal{A}(n_1^2n_2,m,h,q)\,\mathcal{C}(n_1, n_2, m, h; q) + O(X^{-A}),
		\end{align} where $A$ is a large positive real number, $N_0 = {{N}}/{n_1^2},\,  M_0 = {q^2}/{Y},\,H_0 = {q^2}/{H}$, the integral transform $\mathcal{A}(n_1^2n_2,m,h,q)$ is given by
  \begin{align}\label{s5}
      \int_{\mathbb{R}} W(u)\,\psi(q,u)\,\, \mathcal{I}_1(h,u,q)\,\, \mathcal{I}_2(n^2_1n_2,u,q)\,\, \mathcal{I}_3(m,u,q)du,
  \end{align} and the character sum $\mathcal{C}(n_1, n_2, m, h; q)$ is given by
  \begin{align}\label{s6}
   \sideset{}{^\star} \sum_{a\, \mathrm{mod}\, q} e\left( -\frac{\overline{a}h}{q}\right)e\left(\frac{\overline{a}m}{q}\right)  S\left( \bar{a},  \pm n_{2}; q/n_{1}\right).   
  \end{align}
	\end{lemma}
	\begin{proof}
		Using Lemmas \eqref{lemma8}, \eqref{lemma9} and \eqref{lemma10} into the expression for $S(H,X)$ given in equation \eqref{A17}, we get our desired result.
	\end{proof}

	\subsection{Simplification of Integrals\,:}
	After applying summation formulae, we arrive at the fourfold integral denoted as $\mathcal{A}(n_1^2n_2,m,h,q)$, as defined in Equation \eqref{s5}. This section is dedicated to the simplification of this integral and the determination of its bounds. We record the calculations in the following lemma.
	
	\begin{lemma}\label{lemma12}
		We have
		\begin{align*}
			\mathcal{A}(n_1^2n_2,m,h,q)\,\ll \frac{q}{Q}.
		\end{align*}
	\end{lemma} 
	\begin{proof}
		Recalling the expressions of the integrals $\mathcal{I}_1(h,u,q)$,$\mathcal{I}_3(m,u,q)$ and $\mathcal{I}_2(n^2_1n_2,u,q)$ from Lemmas \eqref{lemma8}, \eqref{lemma9} and \eqref{lemma10}, respectively, and plugging into the equation \eqref{s5}, we arrived at
		\begin{align}\label{C1}
			\mathcal{A}(n_1^2n_2,m,h,q)\notag= & \,\int_{\mathbb{R}}W(u)\, \psi(q,u)\,e\left(\frac{(Hx - Yy + Xz)u}{qQ} \right)\\
   \notag&\times\,\int_0^\infty  V_1(x)\, e\left(\pm \frac{2 \sqrt{Hhx}}{q} \right) \int_0^\infty  V_3(y)\, e\left(\pm \frac{2 \sqrt{Ymy}}{q} \right)\,\\ & \times \int_{0}^\infty V_2(z)\,e\left(\pm  \frac{3(Xz n_1^2 n_2)^{1/3}}{q} \right)dx\,dy\,dz\,du. 
		\end{align} 	
We first consider the $u$-integral which is given by,
		\begin{align}\label{A35}
			\int_{\mathbb{R}}W(u)\, \psi(q,u)\,e\left(\frac{(Hx - Yy + Xz)u}{qQ}\right) du.
		\end{align}
		We will analyze the $u$-integral in different cases depending upon the size of variable $q$. For the case of small $q$ i.e. $q \ll Q^{1-\epsilon}$, we split the $u$-integral into two parts, $|u| \ll Q^{-\epsilon}$ and $|u| \gg Q^{-\epsilon}$. By using properties \eqref{delta} of $\psi(q,u)$, for $|u| \ll Q^{-\epsilon}$, we can replace $\psi(q,u)$ by 1 with negligible error term. So essentially we get
		\begin{align*}
			\int_{|u| \ll Q^{-\epsilon}}W(u)\, e\left(\frac{(Hx - Yy + Xz)u}{qQ} \right) du
		\end{align*}
Now, integrating by parts repeatedly, we get that the integral is negligibly small unless
\begin{align}\label{A36}
			|Hx+Xz-Yy| \ll q\,Q^{1+\epsilon} \,\,\, i.e \,\,\,\,\,\, \left|\frac{Hx+Xz}{Y} - y\right| \ll \frac{q\,Q^{1+\epsilon}}{Y} .
		\end{align}
For $|u| \gg Q^{-\epsilon}$, integrating by parts the $u$-integral \eqref{A35} repeatedly and using the properties of involved bump functions i.e.
		\begin{align*}
			\frac{\partial^j}{ \partial u^j} \psi(q, u) \ll \  \min \left\lbrace \frac{Q}{q}, \frac{1}{|u|} \right\rbrace \  \frac{\log Q}{|u|^j}\ll Q^{\epsilon j}, \hspace{0.5cm}	W^j(u) \ll Q^{\epsilon j}
		\end{align*} we will get the same restriction as above in \eqref{A36}. Also, for the case of large $q$ i.e. $q \gg Q^{1-\epsilon}$, condition \eqref{A36} is trivially true.
Let $\frac{Hx+Xz}{Y} - y = t$ with $|t| \ll \frac{qQ^{1+\epsilon}}{Y}$, we reduced integral transform $\mathcal{A}(n_1^2n_2,m,h,q)$ into the following expression
\begin{align}\label{s15}
 &\,\int_{|t| \ll \frac{qQ^{1+\epsilon}}{Y}}\,\int_{0}^\infty\,\,V_{1}(x) \int_{0}^\infty\,V_{2}(z)\,\,V_{3}\left(\frac{Hx+Xz}{Y}-t\right)\, e\left(\pm \frac{2 \sqrt{Hhx}}{q}\right)\\ 
&\,\,\,\,\,\,\,\,\,\,\,\,\,\times\,\,\, e\left(\pm \frac{3(Xz n_1^2 n_2)^{1/3}}{q}\pm\frac{2 \sqrt{m(Hx+Xz-Yt)}}{q} \right) dx\,dz\,dt \notag + O(X^{-A}),
\end{align}
where $A$ is a large positive real number. Now, estimating the integral $\mathcal{A}(n_1^2n_2,m,h,q) $ trivially, we get
\begin{align}\label{A86}
\mathcal{A}(n_1^2n_2,m,h,q)\,  \ll \,\frac{q}{Q}. 
\end{align}
\end{proof}

\subsection{Applying the Cauchy-Schwartz inequality and the Poisson summation formula :} In this subsection, we shall apply the Cauchy-Schwartz inequality and then the Poisson summation formula to sum over $n_2$ given in Lemma \ref{lemma11}. We have 
\begin{align*}
	{S}(H,X) \ll \,&\frac{X^{2/3}Y^{3/4}}{QH^{1/4}}\sum_{q\leq Q}\frac{1}{q^{3}}\, \,\sum_{\pm} \,\sum_{n_{1}|q} \sum_{n_{2} \ll N_0}  \frac{\lambda_{\pi}(n_{2},n_{1})}{n^{-1/3}_{1} n^{1/3}_{2}} \\
	&\times \sum_{h\ll H_0} \sum_{m\ll M_0} \frac{\lambda_f(h)}{h^{1/4}}  \frac{\lambda_g(m)}{m^{1/4}}\,\mathcal{C}(n_1, n_2, m, h; q)\, \mathcal{A}(n^2_1n_2,m,h,q).
\end{align*}
Splitting the sum over $q$ into dyadic blocks  $q \sim C$ with $C \ll Q$ and writing $q = q_1q_2$ with $q_1|(n_1)^{\infty}, (q_2,n_1) =1$. We get
\begin{align*}
	{S}(H,X,C) \ll&\, \frac{X^{2/3}Y^{3/4}}{QH^{1/4}C^{3}}\,\sum_{\pm}\,\sum_{n_1 \ll C}\,  \sum_{n_{1}|q_1|(n_1)^{\infty}} n^{1/3}_1 \sum_{n_{2} \ll N_0}  \frac{|\lambda_{\pi}(n_{2},n_{1})|}{n^{1/3}_{2}} \\
	&\times \, \left|\sum_{q_2\sim C/q_1}\,\, \,\sum_{h\ll H_0} \sum_{m\ll M_0} \frac{\lambda_f(h)}{h^{1/4}}  \frac{\lambda_g(m)}{m^{1/4}}   \,\mathcal{C}(n_1,n_2,m,h;q)\, \mathcal{A}(n^2_1n_2,m,h,q)\right|.
\end{align*} 
For the smooth analysis of the sum ${S}(H,X, C)$, we break the sum over $h,m$ into dyadic blocks $h \sim H_1,\, m \sim M_1$ with $H_1 \ll H_0,\, M_1 \ll M_0$. Applying the Cauchy-Schwartz inequality to the sum over $n_2$, we get

\begin{align}\label{A45}
	{S}(H,X,C) \ll\, \frac{X^{17/12}}{QH^{1/4}C^{3}}\,\sum_{\pm}\,\sum_{n_1 \ll C}\, \sum_{n_{1}|q_1|(n_1)^{\infty}} n^{1/3}_1\,\, \Theta^{1/2}\,\,\Omega^{1/2},
\end{align}  
where 
\begin{align}\label{A46} 
	\Theta =\, \sum_{n_{2} \ll N_0} \frac{\left|\lambda_{\pi}(n_{2},n_{1})\right|^2}{n^{2/3}_{2}}, \hspace{1cm} \text{and}
\end{align} 
\begin{align}\label{A21}
	\Omega =\, \sum_{n_{2} \ll N_0} \left |\sum_{q_2\sim C/q_1}\,\sum_{h\sim H_1} \sum_{m\sim M_1} \frac{\lambda_f(h)\lambda_g(m)}{(hm)^{1/4}} \,\mathcal{C}(n_1,n_2,m,h;q)\, \mathcal{A}(n^2_1n_2,m,h,q)\right|^2.
\end{align} 
First we simplify the character sum $\mathcal{C}(n_1,n_2,m,h;q)$. From equation \eqref{s5}, we have
\begin{align*}
	\mathcal{C}(n_1,n_2,m,h;q) &=\, \sideset{}{^\star} \sum_{a\;\mathrm{mod} \; q} e\left( \frac{\overline{a}(m-h)}{q}\right) S\left( \bar{a},  \pm n_{2}; q/n_{1}\right) \\
	&= \, \sideset{}{^\star} \sum_{a\;\mathrm{mod} \; q} e\left( \frac{\overline{a}(m-h)}{q}\right)\,\sideset{}{^\star} \sum_{\alpha\;\mathrm{mod} \; q/n_1} e\left( \frac{\overline{a}\alpha \pm n_2 \overline{\alpha}}{q/n_1}\right)\\
	&=\, \sideset{}{^\star} \sum_{\alpha\;\mathrm{mod} \; q/n_1}\,\,\, e\left( \pm\frac{\overline{\alpha}n_2}{q/n_1}\right)\,\sideset{}{^\star} \sum_{a\;\mathrm{mod} \; q} e\left( \frac{\overline{a}(n_1 \alpha + m-h)}{q}\right).
\end{align*} Notice that the sum over $a$ is the Ramanujan's sum. We can write
\begin{align}\label{A22}
	\mathcal{C}(n_1,n_2,m,h;q) =\, \sum_{d|q} d\, \mu\left(\frac{q}{d}\right) \sideset{}{^\star} \sum_{\substack{\alpha\; \mathrm{mod}\; q/n_1\\ n_1 \alpha \equiv (h-m)\;\mathrm{mod}\;d}}  e\left( \pm\frac{\overline{\alpha}n_2}{q/n_1}\right).
\end{align}
Now, we open the absolute modulus brackets in the sum \eqref{A21} and split the sum over $n_2$ into dyadic blocks $n_2 \sim N_1$ with $N_1 \ll N_0$. We make the $n_2$-sum smooth by introducing a smooth and compactly supported function $V$, supported on $[1,2]$. We arrive at the following expression 
\begin{align*}
	\Omega \ll &\,\mathop{\mathop{{\sum \sum}}_{q_2\sim C/q_1}}_{q'_2\sim\ C/q_1}\,\sum_{h_1\sim H_1}\,\sum_{h_2\sim H_1}\,\frac{\left|\lambda_f(h_1)\right|\,\left|\lambda_f(h_2)\right|}{(h_1h_2)^{1/4}}\,\sum_{m_1\sim M_1}\,\sum_{m_2\sim M_1}\,\frac{\left|\lambda_g(m_1)\right|\,\left|\lambda_g(m_2)\right|}{(m_1m_2)^{1/4}}\, \left|L(...)\right|\\
 \ll&\,\frac{1}{H^{1/2}_1 M^{1/2}_1}\,\mathop{\mathop{{\sum \sum}}_{q_2\sim C/q_1}}_{q'_2\sim\ C/q_1}\,\sum_{h_1\sim H_1}\,\sum_{h_2\sim H_1}\,{\left|\lambda_f(h_1)\right|^2}\,\sum_{m_1\sim M_1}\,\sum_{m_2\sim M_1}\,{\left|\lambda_g(m_1)\right|^2}\, \left|L(...)\right|,
\end{align*} where we are using the identity $$\left|\lambda_f(h_1)\right|\,\left|\lambda_f(h_2)\right|\,\left|\lambda_g(m_1)\right|\,\left|\lambda_g(m_2)\right| \leq  \left|\lambda_f(h_1)\right|^2\,\left|\lambda_g(m_1)\right|^2 + \left|\lambda_f(h_2)\right|^2\,\left|\lambda_g(m_2)\right|^2,$$ and proceeding with only one such term because both the terms are of same complexity and hence will share the same upper bound. Also, here $L(...)$ is given by 
\begin{align}\label{A24}
	L(...) = \,\notag&\sum_{n_{2} \in \mathbb{Z}} V\left(\frac{n_2}{N_1}\right) \mathcal{C}(n_1,n_2,m_1,h_1;q)\,\overline{\mathcal{C}(n_1,n_2,m_2,h_2;q')}\\ &\times\,\,\mathcal{A}(n^2_1n_2,m_1,h_1,q)\,\overline{\mathcal{A}(n^2_1n_2,m_2,h_2,q')}.
\end{align} Here we are taking $q' = q_1q'_2$. By using equation \eqref{A22}, we can write
\begin{align*}
	\mathcal{C}(n_1,n_2,m_1,h_1;q)\,&\overline{\mathcal{C}(n_1,n_2,m_2,h_2;q')} =\, \sum_{d|q}\sum_{d'|q'} d\,d'\, \mu\left(\frac{q}{d}\right)\mu\left(\frac{q'}{d'}\right) \\ \times\,\, &\sideset{}{^\star} \sum_{\substack{\alpha\; \mathrm{mod}\; q/n_1\\ n_1 \alpha \equiv (h_1-m_1)\;\mathrm{mod}\;d}}\,\,\,\, \sideset{}{^\star}\sum_{\substack{\alpha'\; \mathrm{mod}\; q'/n_1\\ n_1 \alpha' \equiv (h_2-m_2)\;\mathrm{mod}\;d'}}  e\left( \frac{(\overline{\alpha}q'_2 - \overline{\alpha}'q_2)n_2}{q_1q_2q'_2/n_1}\right).
\end{align*} Let $P = \frac{q_1q_2q'_2}{n_1}$ and changing  $ n_2 \longrightarrow \beta + n_2 P $, then equation \eqref{A24} reduces to,
\begin{align*}
	L(...) =&\, \sum_{d|q}\sum_{d'|q'} d\,d'\, \mu\left(\frac{q}{d}\right)\mu\left(\frac{q'}{d'}\right) \sideset{}{^\star} \sum_{\substack{\alpha\; \mathrm{mod}\; q/n_1\\ n_1 \alpha \equiv (h_1-m_1)\;\mathrm{mod}\;d}}\,\,\,\, \sideset{}{^\star}\sum_{\substack{\alpha'\; \mathrm{mod}\; q'/n_1\\ n_1 \alpha' \equiv (h_2-m_2)\;\mathrm{mod}\;d'}}\\
	\times & \sum_{\beta\; \mathrm{mod}\;P}e\left( \frac{(\overline{\alpha}q'_2 - \overline{\alpha}'q_2)\beta}{P}\right) \sum_{n_{2} \in \mathbb{Z}} \,  V\left(\frac{\beta +n_2 P}{N_1}\right) \\
	\times&\,\,\mathcal{A}(n^2_1(\beta +n_2 P),m_1,h_1,q)\,\overline{\mathcal{A}(n^2_1(\beta +n_2 P) ,m_2,h_2,q')}.
\end{align*}
Now, on applying the Poisson summation formula to sum over $n_2$, we get
\begin{align*}
	L(...) =\, & \frac{N_1}{P}\sum_{d|q}\sum_{d'|q'} d\,d'\, \mu\left(\frac{q}{d}\right)\mu\left(\frac{q'}{d'}\right) \sideset{}{^\star} \sum_{\substack{\alpha\; \mathrm{mod}\; q/n_1\\ n_1 \alpha \equiv (h_1-m_1)\;\mathrm{mod}\;d}}\,\,\,\, \sideset{}{^\star}\sum_{\substack{\alpha'\; \mathrm{mod}\; q'/n_1\\ n_1 \alpha' \equiv (h_2-m_2)\;\mathrm{mod}\;d'}}\\
	&\times \, \sum_{n_{2} \in \mathbb{Z}}\,\,\sum_{\beta\; \mathrm{mod}\;P}e\left( \frac{(\overline{\alpha}q'_2 - \overline{\alpha}'q_2+ n_2)\beta}{P}\right) \mathcal{F}(...), 
\end{align*}
where the integral transform $\mathcal{F}(...)$ is given by
\begin{align}\label{A41}
	\mathcal{F}(...) =\, \int_{\mathbb{R}} V(w)\, \mathcal{A}(n^2_1N_1w,m_1,h_1,q)\,\, \overline{\mathcal{A}(n^2_1N_1w ,m_2,h_2,q')}\,\, e\left(-\frac{n_2N_1w}{P}\right)dw.
\end{align}
Finally, we arrived at
\begin{align}\label{A23}
	\Omega \ll&\,\notag \frac{N_1}{H^{1/2}_1 M^{1/2}_1} \sum_{q_2\sim C/q_1}\sum_{q'_2\sim\ C/q_1}\, \,  \,\sum_{h_1\sim H_1}\,\sum_{h_2\sim H_1}{\left|\lambda_f(h_1)\right|^2}\\
 &\times\,\sum_{m_1\sim M_1}\,\sum_{m_2\sim M_1}{\left|\lambda_g(m_1)\right|^2}\, \sum_{n_2\in \mathbb{Z}} \,\mathfrak{C}(...)\,\mathcal{F}(...),
\end{align} where $\mathcal{F}(...)$ is same as defined in equation \eqref{A41} and the character sum is given as
\begin{align}\label{A71}
	\mathfrak{C}(...) = \sum_{d|q}\sum_{d'|q'} d\,d'\, \mu\left(\frac{q}{d}\right)\mu\left(\frac{q'}{d'}\right) \mathop{\sideset{}{^\star} {\sum}_{\substack{\alpha\; \mathrm{mod}\; q/n_1\\ n_1 \alpha \equiv (h_1-m_1)\;\mathrm{mod}\;d}}\,\,\,\, \sideset{}{^\star}\sum_{\substack{\alpha'\; \mathrm{mod}\; q'/n_1\\ n_1 \alpha' \equiv (h_2-m_2)\;\mathrm{mod}\;d'}}}_{\overline{\alpha}q'_2 - \overline{\alpha}'q_2 \equiv -n_2\; \mathrm{mod}\;q_1q_2q'_2/n_1} 1.
\end{align}
	
\subsection{Estimation for the zero frequency:} 
The case of zero frequency, where $n_2 = 0$, needs a distinct treatment. We will denote the contributions of the zero frequency case to the character sum $\mathfrak{C}$ (as given in \eqref{A71}), the integral transform $\mathcal{F}$ (as given in \eqref{A41}), and $\Omega$ (as given in \eqref{A23}) as $\mathfrak{C}_0(...)$, $\mathcal{F}_0(...)$, and $\Omega_0$ respectively. Finally, we will estimate the contribution of zero frequency to our main term $S(H, X,C)$ as defined in \eqref{A45}, and we will denote it as $S_{0}(H, X, C)$.
	
\subsubsection{Bound for the character sum $(n_2 =0)$:}	For the zero frequency case i.e. $n_2 =0$, the congruence condition modulo $P = {q_1q_2q'_2}/{n_1}$ in the character sum $\mathfrak{C}$ becomes
	\begin{align*}
		\overline{\alpha}q'_2 - \overline{\alpha}'q_2 \equiv 0\; \left(\frac{q_1q_2q'_2}{n_1}\right).
	\end{align*} Reducing this congruence condition modulo $q_2$ and $q'_2$, will get $q_2 = q'_2$ and $ \alpha = \alpha'$. So we can write
	\begin{align*}
		\mathfrak{C}_0(...) \leq\, \sum_{d|q}\sum_{d'|q}\,\, d\,d'\sideset{}{^\star} {\sum}_{\substack{\alpha\; \mathrm{mod}\; q/n_1\\ n_1 \alpha \equiv (h_1-m_1)\;\mathrm{mod}\;d \\  n_1 \alpha \equiv (h_2-m_2)\;\mathrm{mod}\;d' }} 1 \,&\leq \,\mathop{\sum_{d|q}\sum_{d'|q} }_ {(d,d')\, |\, (m_1-m_2)-(h_1-h_2)} d\,d' \, \frac{q}{n_1[d,d']}.\end{align*} 
  Finally, we get
	\begin{align}\label{A42}
		 \mathfrak{C}_0(...) \leq \frac{q}{n_1}\mathop{\sum_{d|q}\sum_{d'|q} }_ {(d,d')\, |\, (m_1-m_2)-(h_1-h_2)} (d,d'),
	\end{align}
	where $(d, d')$ and $[d, d']$ represent the greatest common divisor and least common multiple of $d, d'$ respectively.

\subsubsection{Bound for the integral transform $(n_2 =0)$:}  For the case $n_2 =0$, by estimating trivially the integral transform $\mathcal{F}(...)$ given in \eqref{A41} using lemma \eqref{lemma12}, we get the following bound
	\begin{align}\label{A43}
		\mathcal{F}_{0}(...) \, \ll  \, \frac{C^2}{Q^2}.
	\end{align}

In the next lemma, we will obtain the contribution of zero frequency case to our main sum $S(H, X,C)$ given in the equation \eqref{A45}.
 
\begin{lemma}\label{lemma13}
	For the zero frequency case i.e. $n_2 = 0$, we have,
\begin{align}
 S_{0}(H,X, C) \ll\, \frac{X^{3/4}\,\,}{H^{1/2}\,}\left(1 + \frac{Q}{H}\right)^{1/2}.
\end{align}  
\end{lemma}
\begin{proof}
 Plugging $n_2 =0$ in equation \eqref{A23}, we have 
\begin{align*}
	\Omega_0 \ll&\,\notag \frac{N_1}{H^{1/2}_1 M^{1/2}_1} \sum_{q_2\sim C/q_1}\,  \,\sum_{h_1\sim H_1}\,\sum_{h_2\sim H_1}{\left|\lambda_f(h_1)\right|^2}\,\sum_{m_1\sim M_1}\,\sum_{m_2\sim M_1}{\left|\lambda_g(m_1)\right|^2} \,\mathfrak{C}_0(...)\,\mathcal{F}_0(...).
\end{align*}	 
Using the bounds of the character sum $\mathfrak{C}_0$ given in \eqref{A42} and integral transform $\mathcal{F}_{0}(...)$ given in \eqref{A43}, we get
\begin{align*}
\Omega_0 \ll \frac{N_1 C^2}{H^{1/2}_1 M^{1/2}_1 Q^2 n_1} \sum_{q_2\sim C/q_1}q \sum_{d|q}\sum_{d'|q} (d,d') \mathop{\sum_{h_1\sim H_1}\sum_{h_2\sim H_1}\sum_{m_1\sim M_1}\sum_{m_2\sim M_1}}_{(d,d')|\, (m_1-m_2)-(h_1-h_2)}\,{\left|\lambda_f(h_1)\right|^2}{\left|\lambda_g(m_1)\right|^2}.
\end{align*}
Considering the four different cases i.e. $ (1): m_1 = m_2, h_1 = h_2$,\, $ (2): m_1 = m_2, h_1 \neq h_2$,\, $ (3): m_1 \neq m_2, h_1 = h_2$ and $ (4): m_1 \neq m_2, h_1 \neq h_2$, we arrived at  
\begin{align*}
	\Omega_0 \ll&\, \frac{N_1\,C^3}{H^{1/2}_1\, M^{1/2}_1\,Q^2\,n_1} \sum_{q_2\sim C/q_1} \,\sum_{d|q}\sum_{d'|q} (d,d')\,\left(M_1H_1 + \frac{M_1 H^2_1}{(d,d')} + \frac{H^2_1 M_1}{(d,d')} + \frac{H^2_1M^2_1}{(d,d')}\right)\\
\ll&\,\frac{N_1\,H^{1/2}_1\, M^{1/2}_1\,C^3}{\,Q^2\,n_1}\,\sum_{q_2\sim C/q_1} \,\sum_{d|q}\sum_{d'|q} \left((d,d') + M_1 + H_1 + {H_1M_1} \right)\\
\ll&\, \frac{N_0\,H^{1/2}_0\,M^{1/2}_0\, C^3}{\,Q^2\,n_1}\,\sum_{q_2\sim C/q_1} \,\left(q + M_0 + H_0 + {H_0M_0} \right).
\end{align*} So, finally we can write
\begin{align*}
	\Omega_0 \ll \frac{N\,H^{1/2}_0\,M^{1/2}_0\,C^4}{\,Q^2\,n^3_1\,q_1}\,\left(C + M_0{H_0} \right).
\end{align*}
Now we put this bounds of $\Omega_0$ into our main term given in equation \eqref{A45}, we get
\begin{align*}
 S_{0}(H,X,C) \ll&\, \frac{X^{17/12}}{Q\,H^{1/4}\,C^{3}}\,\sum_{n_1 \ll C}\, \sum_{n_{1}|q_1|(n_1)^{\infty}} n^{1/3}_1\,\, \Theta^{1/2}\,\,\	(\Omega_0)^{1/2}\\
 \ll& \,\frac{X^{17/12}}{Q\,H^{1/4}\,C^{3}}\,\sum_{n_1 \ll C}\, \sum_{n_{1}|q_1|(n_1)^{\infty}} n^{1/3}_1\,\Theta^{1/2}\,\left(\frac{N\,H^{1/2}_0\,M^{1/2}_0\, C^4}{\,Q^2\,n^3_1\,q_1}\,\left(C + M_0{H_0} \right)\right)^{1/2}\\
 \ll&\, \frac{X^{17/12}\,N^{1/2}\,H^{1/4}_0\,M^{1/4}_0\,}{Q^{2}\,H^{1/4}\,C}\,\left(C + M_0{H_0} \right)^{1/2}\,\sum_{n_1 \ll C}\, \sum_{n_{1}|q_1|(n_1)^{\infty}}\frac{n^{1/3}_1}{n^{3/2}_1\,\sqrt{q_1}}\,\Theta^{1/2}\\
\ll&\, \frac{X^{17/12}\,N^{1/2}\,H^{1/4}_0\,M^{1/4}_0}{Q^{2}\,H^{1/4}\,C}\,\left(C + M_0{H_0} \right)^{1/2}\,\sum_{n_1 \ll C}\,\frac{\Theta^{1/2}}{n^{5/3}_1}.
\end{align*}
Note that, for $k \geq 7/6$, and using the value of $\Theta$ from equation \eqref{A46}, it's a simple observation that by applying Cauchy-Schwartz inequality to $n_1$ sum, we can get
\begin{align}\label{s10}
\sum_{n_1 \ll q}\frac{\Theta^{1/2}}{n^{k}_1}\,\ll\, \left(\sum_{n_1 \ll C}\, \frac{1}{n^{2k-4/3}_1}\right)^{1/2}\, \left(\sum_{n^2_1n_{2}\ll N}  \frac{|\lambda_{\pi}(n_{2},n_{1})|^2}{(n^2_1n_2)^{2/3}} \right)^{1/2} \,\ll\, N^{1/6}.    \end{align} By using the above estimates, we get
\begin{align*} 
S_{0}(H,X,C)\,\ll&\, \frac{X^{17/12}\,N^{2/3}\,H^{1/4}_0\,M^{1/4}_0}{Q^{2}\,H^{1/4}\,C}\,\left(C + M_0{H_0} \right)^{1/2}.
\end{align*}
Now plugging the values
$N  \ll {Q^3}/{X}, \hspace{0.5cm} M_0 \asymp {C^2}/{X}, \hspace{0.5cm} H_0 \asymp {C^2}/{H}$, we finally get
\begin{align*}
S_{0}(H,X,C) \ll \frac{X^{3/4}\,\,}{H^{1/2}\,}\left(1 + \frac{Q}{H}\right)^{1/2}.
\end{align*}
This is our desired result.
\end{proof}

\vspace{0.1cm} 
\subsection{Estimation for the non-zero frequencies:}\label{nonzero} We will refer to the contribution associated with non-zero frequencies, i.e., when $n_2\neq 0$, in the character sum $\mathfrak{C}(...),$ the integral transform $\mathcal{F}(...)$ and the function $\Omega$ as $\mathfrak{C}_{\neq 0}(...)$, $\mathcal{F}_{\neq 0}(...)$ and $\Omega_{\neq 0}$, respectively. In this section, we will first analyze the character sum $\mathfrak{C}_{\neq 0}(...)$ and then examine the integral transform $\mathcal{F}_{\neq 0}(...)$ in the following two lemmas. Finally, we will estimate our main sum $S(H, X,C)$ as given in \eqref{A45}, which we will denote as $S_{\neq 0}(H, X,C)$.
\vspace{0.1cm}

\begin{lemma}\label{lemma15}
	For the non-zero frequency case i.e. $n_2 \neq 0$, we have 
	\begin{align*}
		\mathfrak{C}_{\neq 0}(...)  \ll \,{q_1}\,(q_2,\, n_1 q'_2 + l_1\,n_2)\,\mathop{\sum_{\substack{\alpha\; \mathrm{mod}\; q_2}}\,\sum_{\substack{\alpha'\; \mathrm{mod}\; q'_2}}}_{\overline{\alpha}q'_2 - \overline{\alpha}'q_2 \equiv -n_2\; \mathrm{mod}\;q_2q'_2} \,\sum_{\substack{\beta'\; \mathrm{mod}\; \frac{q_1}{n_1}}}\,\mathop{\mathop{\sum_{d'_2|q'_2}\,\sum_{d'_1|q_1}}_{n_1 \alpha' \equiv\, l_2\;\mathrm{mod}\;d'_2}}_{n_1 \beta' \equiv\, l_2\;\mathrm{mod}\;d'_1} d'_2 d'_1.
	\end{align*}
 For the simplicity, here we are taking $l_1 = h_1-m_1$ and $l_2 = h_2-m_2$.
\end{lemma}
\begin{proof}
Recall the expression of the character sum from equation \eqref{A71} with  $l_1 = h_1-m_1$ and $l_2 = h_2-m_2$.  we have
	\begin{align*}
		\mathfrak{C}_{\neq 0}(...) = \sum_{d|q}\sum_{d'|q'} d\,d'\, \mu\left(\frac{q}{d}\right)\mu\left(\frac{q'}{d'}\right) \mathop{\sideset{}{^\star} {\sum}_{\substack{\alpha\; \mathrm{mod}\; q_1q_2/n_1\\ n_1 \alpha \equiv\, l_1\;\mathrm{mod}\;d}}\,\,\,\, \sideset{}{^\star}\sum_{\substack{\alpha'\; \mathrm{mod}\; q_1q_2'/n_1\\ n_1 \alpha' \equiv \,l_2\;\mathrm{mod}\;d'}}}_{\overline{\alpha}q'_2 - \overline{\alpha}'q_2 \equiv -n_2\; \mathrm{mod}\;q_1q_2q'_2/n_1} 1.
	\end{align*}
	As we have chosen $(q_2,n_1) = 1$ with $n_1 | q_1$. So, we can write
	\begin{align}\label{A61}
		\mathfrak{C}_{\neq 0}(...) \leq \mathfrak{C}_1(...)\,\mathfrak{C}_2(...),
	\end{align} where 
	\begin{align*}
		\mathfrak{C}_1(...) = \sum_{d_1|q_1}\sum_{d'_1|q_1} d_1\,d'_1\, \mathop{\sideset{}{^\star} {\sum}_{\substack{\alpha_1\; \mathrm{mod}\; q_1/n_1\\ n_1 \alpha_1 \equiv \,l_1\;\mathrm{mod}\;d_1}}\,\,\,\, \sideset{}{^\star}\sum_{\substack{\alpha'_1\; \mathrm{mod}\; q_1/n_1\\ n_1 \alpha'_1 \equiv\, l_2\;\mathrm{mod}\;d'_1}}}_{\overline{\alpha_1}q'_2 - \overline{\alpha'_1}q_2 \equiv -n_2\; \mathrm{mod}\;q_1/n_1} 1,
	\end{align*} and
	\begin{align*}
		\mathfrak{C}_2(...) = \sum_{d_2|q_2}\sum_{d'_2|q'_2} d_2\,d'_2\, \mathop{\sideset{}{^\star} {\sum}_{\substack{\alpha_2\; \mathrm{mod}\; q_2\\ n_1 \alpha_2 \equiv\, l_1\;\mathrm{mod}\;d_2}}\,\,\,\, \sideset{}{^\star}\sum_{\substack{\alpha'_2\; \mathrm{mod}\; q'_2\\ n_1 \alpha'_2 \equiv\, l_2\;\mathrm{mod}\;d'_2}}}_{\overline{\alpha_2}q'_2 - \overline{\alpha'_2}q_2 \equiv -n_2\; \mathrm{mod}\;q_2q'_2} 1.
	\end{align*}
Firstly, consider the congruence condition involving $n_2$ in character sum $\mathfrak{C}_1(...)$. Note that for a given $\alpha_1$, there can be at most one $\alpha'_1$ i.e. we can determine $\alpha'_1$ in terms of $\alpha_1$. Hence, we can write
	\begin{align*}
		\mathfrak{C}_1(...) \ll \sum_{d_1|q_1}\,d_1\,\sum_{d'_1|q_1} \,d'_1\, \mathop{\sum_{\substack{\alpha'_1\; \mathrm{mod}\; q_1/n_1\\ n_1 \alpha'_1 \equiv\, l_2\;\mathrm{mod}\;d'_1}}} 1\,\, \ll \,d(q_1)\,q_1 \sum_{d'_1|q_1} \,d'_1\, \mathop{\sum_{\substack{\alpha'_1\; \mathrm{mod}\; q_1/n_1\\ n_1 \alpha'_1 \equiv\, l_2\;\mathrm{mod}\;d'_1}}} 1.
	\end{align*}
For the another character sum $\mathfrak{C}_2(...)$, by using $(n_1, q_2q'_2) =1$ and the congruence conditions modulo $d_2$, we can write $\alpha_2 \equiv \overline{n_1}l_1\;\mathrm{mod}\;d_2$. Using this determined value of $\alpha_2$ in the congruence condition modulo $q_2q'_2$, we obtained $ n_1\, q'_2 + l_1\,n_2 \equiv 0 \;\mathrm{mod}\;d_2.$ With all these details, we can write
\begin{align*}
\mathfrak{C}_2(...) &\ll\,\mathop{\sum}_{d_2|(q_2,\, n_1 q'_2 + l_1 n_2)} d_2\,\sum_{d'_2|q'_2} d'_2 \,\mathop{\sum_{\substack{\alpha_2\; \mathrm{mod}\; q_2}}\,\mathop{\sum_{\substack{\alpha'_2\; \mathrm{mod}\; q'_2}}}_{n_1 \alpha'_2 \equiv \,l_2\;\mathrm{mod}\;d'_2}}_{\overline{\alpha_2}q'_2 - \overline{\alpha'_2}q_2 \equiv -n_2\; \mathrm{mod}\;q_2q'_2} \,1 \\
&\ll\,d(q_2)\,(q_2,\, n_1 q'_2 + l_1 n_2)\,\sum_{d'_2|q'_2} d'_2 \,\mathop{\sum_{\substack{\alpha_2\; \mathrm{mod}\; q_2}}\,\mathop{\sum_{\substack{\alpha'_2\; \mathrm{mod}\; q'_2}}}_{n_1 \alpha'_2 \equiv \,l_2\;\mathrm{mod}\;d'_2}}_{\overline{\alpha_2}q'_2 - \overline{\alpha'_2}q_2 \equiv -n_2\; \mathrm{mod}\;q_2q'_2} \,1.
\end{align*}
Since we know that $d(n)\, \ll \,n^{\epsilon}$. By using the bounds for $\mathfrak{C}_1(...)$, $\mathfrak{C}_2(...)$ and $d(q_1), d(q_2)$ into equation \eqref{A61}, we will get our desired result.
\end{proof}
\begin{lemma}\label{L1}
	Let $\mathcal{F}(...)$ be as given in \eqref{A41}. Then for the case $n_2 \neq 0$, we have
	\begin{align*}
		\mathcal{F}_{\neq 0}(...) \ll \, \frac{C^2}{Q^2}.
	\end{align*} Also the integral transform $\mathcal{F}_{\neq 0}(...)$ will be negligibly small unless
	\begin{align*}
		|n_2| \ll  \frac{X^{1/3}\,C}{N^{2/3}_1\,q_1\,n^{1/3}_1} := N_2.
	\end{align*}
\end{lemma}
\begin{proof}
From equation \eqref{A41}, we have
	\begin{align}\label{C2}
		\mathcal{F}_{\neq 0}(...) = \int_{\mathbb{R}} V(w)\, \mathcal{A}(n^2_1N_1w,m_1,h_1,q)\,\overline{\mathcal{A}(n^2_1N_1w ,m_2,h_2,q')}\, e\left(-\frac{n_2N_1w}{P}\right)dw.
	\end{align}
	Estimating the above integral trivially by using Lemma \eqref{lemma12}, we will get our desired bound.\\
 
For the proof of the second part of lemma, plugging the expression of the integral $\mathcal{A}(n^2_1N_1w,m_1,h_1,q)$ given in equation \eqref{s15} into the expression of $\mathcal{F}_{\neq 0}(...)$ and considering only the $w$-integral, we have
\begin{align*}
		\int_{\mathbb{R}} V(w)\, e\left(\pm \frac{3(XN_1n^2_1z_1)^{1/3}\,w^{1/3}}{q} \mp \frac{3(XN_1n^2_1z_2)^{1/3}\,w^{1/3}}{q'} \right)\,e\left(-\frac{n_2N_1w}{P}\right)\,dw. 
\end{align*} 
	On applying integration by parts repeatedly to $w$-integral, we will get that it will be negligibly small unless
	\begin{align*}
		|n_2| \ll \frac{P}{N_1}\,\frac{(XN_1n^2_1)^{1/3}}{C} \asymp  \frac{X^{1/3}\,C}{N^{2/3}_1\,q_1\,n^{1/3}_1} := N_2.
	\end{align*} This is our desired result. 
\end{proof}
\begin{lemma}\label{lemma16} 
For the non-zero frequency case i.e. $n_2 \neq 0$, we have
\begin{align*}
S_{\neq 0}(H,X,C) \ll \frac{X}{H^{1/2}}\left(1 + \frac{Q}{H}\right)^{1/2}.
\end{align*}
\end{lemma}
\begin{proof}
For $n_2 \neq 0$, recall the expression for $\Omega$ from equation \eqref{A23}, we have
\begin{align*}
\Omega_{\neq 0}\, \ll&\,\,\notag \frac{N_1}{H^{1/2}_1 M^{1/2}_1} \sum_{q_2\sim C/q_1}\sum_{q'_2\sim\ C/q_1}\, \,  \,\sum_{h_1\sim H_1}\,\sum_{h_2\sim H_1}{\left|\lambda_f(h_1)\right|^2}\\
 &\,\,\,\,\,\,\,\,\,\,\,\,\times\,\sum_{m_1\sim M_1}\,\sum_{m_2\sim M_1}{|\lambda_g(m_1)|^2}\, \sum_{n_2\in \mathbb{Z}} \,\mathfrak{C}_{\neq 0}(...)\,\mathcal{F}_{\neq 0}(...).
\end{align*}
Plugging the bounds for the character sum $\mathfrak{C}_{\neq 0}(...)$ given in Lemma \eqref{lemma15} and the integral transform $\mathcal{F}_{\neq 0}(...)$ given in equation \eqref{L1}, we get
\begin{align}\label{s90}
	\Omega_{\neq 0} \ll& \frac{N_1\,C^2\,q_1}{H^{1/2}_1 M^{1/2}_1\,Q^2}\sum_{q_2\sim C/q_1}\,\sum_{q'_2\sim\ C/q_1}\,\sum_{0<|n_2|\ll N_2}\,\notag\\
 &\times\,\sum_{h_1\sim H_1}\,|\lambda_f(h_1)|^2\,\sum_{m_1\sim M_1}\, |\lambda_g(m_1)|^2\,
(q_2,\, n_1 q'_2 + (h_1-m_1)\,n_2)\notag\\
&\times\,\mathop{\sum_{\substack{\alpha_2\; \mathrm{mod}\; q_2}}\,\sum_{\substack{\alpha'_2\; \mathrm{mod}\; q'_2}}}_{\overline{\alpha_2}q'_2 - \overline{\alpha'_2}q_2 \equiv -n_2\; \mathrm{mod}\;q_2q'_2} \,\sum_{\substack{\alpha'_1\; \mathrm{mod}\; \frac{q_1}{n_1}}}\,\sum_{d'_2|q'_2}\,\sum_{d'_1|q_1} d'_2 d'_1\,\mathop{\mathop{\sum_{h_2 \sim H_1}\,\sum_{m_2 \sim M_1}}_{n_1 \alpha'_2 \equiv\, (h_2-m_2)\;\mathrm{mod}\;d'_2}}_{n_1 \alpha'_1 \equiv\, (h_2-m_2)\;\mathrm{mod}\;d'_1}\,1.
\end{align}
Next, we count $h_2$ and $m_2$ with the help of corresponding congruence relations. Change the variable $l_2 = h_2 -m_2$ with $l_2 \ll \,H_0 + M_0 \ll H_0$, we obtained
\begin{align*}
\mathop{\mathop{\sum_{h_2 \sim H_1}\,\sum_{m_2 \sim M_1}}_{n_1 \alpha'_2 \equiv\, (h_2-m_2)\;\mathrm{mod}\;d'_2}}_{n_1 \alpha'_1 \equiv\, (h_2-m_2)\;\mathrm{mod}\;d'_1}\,1\,\ll\,\sum_{m_2 \sim M_1}\, \mathop{\mathop{\sum_{l_2 \ll H_0}}_{n_1 \alpha'_2 \equiv\, l_2\;\mathrm{mod}\;d'_2}}_{n_1 \alpha'_1 \equiv\, l_2\;\mathrm{mod}\;d'_1}\,\ll\, M_1\,\left(1 + \frac{H_0\,}{d'_2\,d'_1}\right).   
\end{align*}
Now, estimating the $d_2$ and $d'_2$ sums in equation \eqref{s90} with above the estimate, we arrived at
\begin{align*}
	\Omega_{\neq 0} \ll\,& \frac{N_1\,C^2\,q_1\,M^{1/2}_1}{H^{1/2}_1 \,Q^2}\sum_{q_2\sim C/q_1}\,\sum_{q'_2\sim\ C/q_1}\,\sum_{0<|n_2|\ll N_2}\,\sum_{h_1\sim H_1}\,|\lambda_f(h_1)|^2\,\sum_{m_1\sim M_1}\, |\lambda_g(m_1)|^2\notag\\
 &\,\,\,\,\,\,\,\,\,\,\,\,\,\,\,\times\,(q_2,\, n_1 q'_2 + (h_1-m_1)\,n_2)\mathop{\sum_{\substack{\alpha_2\; \mathrm{mod}\; q_2}}\,\sum_{\substack{\alpha'_2\; \mathrm{mod}\; q'_2}}}_{\overline{\alpha_2}q'_2 - \overline{\alpha'_2}q_2 \equiv -n_2\; \mathrm{mod}\;q_2q'_2} \,\sum_{\substack{\alpha'_1\; \mathrm{mod}\; \frac{q_1}{n_1}}}\,\left(q_1q'_2 + H_0\right).
\end{align*}
Executing the sums over $\alpha_2,\, \alpha'_2$ and $\alpha'_1$, we obtained
\begin{align*}
\Omega_{\neq 0} \ll\,& \frac{N_1\,C^2\,q^2_1\,M^{1/2}_1}{H^{1/2}_1 \,Q^2\,n_1}\,\sum_{q'_2\sim\ C/q_1}\,\left(q_1q'_2 + H_0\right)\,\sum_{h_1\sim H_1}\,|\lambda_f(h_1)|^2\notag\\
 &\,\,\,\times\,\sum_{m_1\sim M_1}\, |\lambda_g(m_1)|^2\,\sum_{0<|n_2|\ll N_2}\,\sum_{q_2\sim C/q_1}\,(q_2, n_2) \,(q_2,\, n_1 q'_2 + (h_1-m_1)\,n_2).
\end{align*}
Notice that the sum over $q_2$ is bounded by $(n_2, n_1q'_2) C/q_1$, which again summed over $n_2$ is bounded by $N_2 C/q_1$. Now we arrived at
\begin{align*}
\Omega_{\neq 0} \ll\,& \frac{N_1\,C^2\,q^2_1\,M^{1/2}_1}{H^{1/2}_1 \,Q^2\,n_1}\,\frac{N_2 C}{q_1}\,\sum_{q'_2\sim\ C/q_1}\,\left(q_1q'_2 + H_0\right)\,\sum_{h_1\sim H_1}\,|\lambda_f(h_1)|^2\,\sum_{m_1\sim M_1}\, |\lambda_g(m_1)|^2. 
\end{align*}
Estimating the sum over $h_1,\,m_1$ using Ramanujan bound on average \eqref{s16}, and executing $q'_2$ sum. Also using the value of $N_2$ from Lemma \ref{L1}, we get
\begin{align*}
\Omega_{\neq 0} \ll\,& \frac{N_1\,C^2\,q^2_1\,M^{1/2}_1}{H^{1/2}_1 \,Q^2\,n_1}\,\frac{X^{1/3}\,C}{N^{2/3}_1\,q_1\,n^{1/3}_1}\,\frac{C^2}{q^2_1}\, M_1\,H_1\,\left(C + H_0\right).     
\end{align*}
With $N_1 \ll N_0 = N/n^2_1$, we finally obtained the following bound 
\begin{align*}
\Omega_{\neq 0} \ll&\, \frac{X^{1/3}\,N^{1/3}\, C^5\, M^{3/2}_0\,H^{1/2}_0}{Q^2\,n^2_1\,q_1} (C+ H_0) .
\end{align*}

Now, using the bound of $\Omega_{\neq 0}$ into the expression of our main term $S_{\neq 0}(H,X,C)$ given in equation \eqref{A45}, we obtained that our main term is dominated by
\begin{align*}
&\,\frac{X^{17/12}}{Q\,H^{1/4}\,C^{3}}\sum_{n_1 \ll C}\, \sum_{n_{1}|q_1|(n_1)^{\infty}} n^{1/3}_1\,\Theta^{1/2}\,\left(\frac{X^{1/3}\,N^{1/3}\, C^5\, M^{3/2}_0\,H^{1/2}_0}{Q^2\,n^2_1\,q_1} (C+ H_0)\right)^{1/2}\\
\ll&\, \frac{X^{19/12}\,N^{1/6}\,H^{1/4}_0\,M^{3/4}_0}{Q^{2}\,H^{1/4}\,C^{1/2}}\,\left({C}+ {H_0}\right)^{1/2}\,\sum_{n_1 \ll C}\, \sum_{n_{1}|q_1|(n_1)^{\infty}}\,\frac{n^{1/3}_1}{n_1\,q^{1/2}_1}\,\Theta^{1/2}\\
\ll& \, \frac{X^{19/12}\,N^{1/6}\,H^{1/4}_0\,M^{3/4}_0}{Q^{2}\,H^{1/4}\,C^{1/2}}\,\left({C}+ {H_0}\right)^{1/2}\,\sum_{n_1 \ll C}\, \frac{\Theta^{1/2}}{n^{7/6}_1}.
\end{align*}
For $n_1$-sum, using the result given in equation \eqref{s10}, we get
\begin{align*}
S_{\neq 0}(H,X,C)\,\ll& \,\frac{X^{19/12}\,N^{1/3}\,H^{1/4}_0\,M^{3/4}_0}{Q^{2}\,H^{1/4}\,C^{1/2}}\,\left({C}+ {H_0}\right)^{1/2}.
\end{align*}
Now, by using 
$N  \ll {Q^3}/{X}, \hspace{0.5cm} M_0 \asymp {C^2}/{Y}, \hspace{0.5cm} H_0 \asymp {C^2}/{H},$  we get,
\begin{align*}
S_{\neq 0}(H,X,C) \,\ll\, \frac{X}{H^{1/2}}\left(1 + \frac{Q}{H}\right)^{1/2}.
\end{align*} which is desired result.
\end{proof}

\vspace{0.2cm}

\subsection{Final Estimates:} As we have,
\begin{align*}
S(H,X) \ll\,   S_{0}(H,X,C) + S_{\neq 0}(H,X,C).
\end{align*} Using the bounds of $S_{0}(H,X,C)$ from Lemma \eqref{lemma13} and $S_{\neq 0}(H,X,C)$ from Lemma \eqref{lemma16}, we get
\begin{align}\label{fe}
	S(H,X) \ll&\, \,\frac{X^{3/4}\,\,}{H^{1/2}\,}\left(1 + \frac{Q}{H}\right)^{1/2} + \frac{X}{H^{1/2}}\left(1 + \frac{Q}{H}\right)^{1/2}.
\end{align}
With $Q \asymp \sqrt{X}$, and $ X^{1/4+\delta} \leq H \leq \sqrt{X}$ with $\delta >0$, we get
\begin{align*}
S(H,X) \ll\, \frac{X^{5/4+\epsilon}}{H}. 
\end{align*} 
This proves the first part of Theorem \ref{main theorem}. Also, in the range $ \sqrt{X} \leq H \leq X$, we have $Q/H \ll 1$. Hence, from equation \eqref{fe}, we get
\begin{align*}
S(H,X) \ll\, \frac{X}{H^{1/2}}.    
\end{align*}This proves the second part of Theorem \ref{main theorem}.\\

\section{Sketch of proof (theorem 2)}
This section will give a brief sketch of the proof of Theorem \ref{TH2}. Now our main object of study is
\begin{align*}
	\mathcal{L}(H,X) =  \frac{1}{H}\sum_{h=1}^\infty \lambda_f(h) V_1\left( \frac{h}{H}\right)\sum_{n=1}^\infty  d_3(n)  \lambda_g (n+h) V_2\left( \frac{n}{X} \right),
\end{align*} where $d_3(n)$ is the triple divisor function. By applying the delta method given in \eqref{Delta}, we arrive at the following expression
\begin{align}\label{B21}
	\mathcal{L}(H,X)=\,\notag& \frac{1}{HQ}\sum_{q\leq Q}\frac{1}{q}\,\,\, \sideset{}{^\star} \sum_{a\bmod q}\int_{\mathbb{R}} W(u)\psi(q,u)\, \sum_{h=1}^\infty \lambda_f(h)  \,e\left(\frac{ah}{q}\right)\,e\left(\frac{hu}{qQ}\right) V_1\left( \frac{h}{H}\right)\\ 
	\notag&\times\,\sum_{n=1}^\infty  d_3(n)\, e\left(\frac{an}{q}\right)\,e\left(\frac{nu}{qQ}\right) V_2\left( \frac{n}{X}\right)\\
	&\times\,\sum_{m=1}^\infty   \lambda_g (m)\,e\left(\frac{-am}{q}\right)\,e\left(\frac{-mu}{qQ}\right)  V_3\left( \frac{m}{Y}\right)\,du.
\end{align} 
Next, we apply the $GL(2)$ and $GL(3)$ Voronoi summation formulae for $h,m$ and $n$ sums, respectively. We first recall the Voronoi summation formula for $d_3(n)$.
Let 
$$\sigma_{0,0}(k_1,k_2)=\sum_{d_1|k_2}\mathop{\sum}_{\substack{{d_2d_1 | k_2} \\ (d_2,k_1)=1}}1.$$
Let $\phi$ be a compactly supported smooth function on  $ (0, \infty )$ and $\tilde{\phi}(s)$ be its Mellin transform. For $\ell= 0$ and $1$, we define
\begin{equation*}
	\gamma_{\ell}(s) :=  \frac{\pi^{-3s-\frac{3}{2}}}{2} \, \left( \frac{\Gamma\left(\frac{1+s+ \ell}{2}\right)}{\Gamma\left(\frac{-s+ \ell}{2}\right)}\right)^3.
\end{equation*}
Set $\gamma_{\pm}(s) = \gamma_{0}(s) \mp \gamma_{1}(s)$ and let 
\begin{align}\label{gl3 integral transform for d_3}
	G_{\pm}(y) = \frac{1}{2 \pi i} \int_{(\sigma)} y^{-s} \, \gamma_{\pm}(s) \, \tilde{g}(-s) \, \mathrm{d}s,
\end{align}
where $\sigma > -1$. By using the above terminology, we now state the $GL(3)$ Voronoi summation formula for $d_3$ in the next lemma.
\begin{lemma} \label{lemma18}
	Let $\phi(x)$ and  $d_3(n)$ be defined as above. Let $a, \bar{a}, q \in \mathbb{Z}$ with $(a,q)=1,$ and  $a\bar{a} \equiv 1(\mathrm{mod} \ q)$. Then we have
	\begin{align} \label{GL3-Voro for d-3}
		\sum_{n=1}^{\infty}& d_3(n) e\left(\frac{an}{q}\right) \phi(n) \notag  \\
		&=\,\frac{1}{2q^2}\widetilde{\phi}(1) \sum_{n_{1}|q} n_1\tau(n_1)P_2(n_1,q) S\left( \bar{a}, 0; \frac{q}{n_1}\right) \notag \\
		&+\,\frac{1}{2q^2}\widetilde{\phi}^\prime(1) \sum_{n_{1}|q} n_1\tau(n_1)P_1(n_1,q) S\left( \bar{a}, 0; \frac{q}{n_1}\right) \notag \\
		&+\frac{1}{4q^2}\widetilde{\phi}^{\prime \prime}(1) \sum_{n_{1}|q} n_1\tau(n_1) S\left( \bar{a}, 0; \frac{q}{n_1}\right)  \notag \\
		&+\,{q} \sum_{\pm} \sum_{n_{1}|q} \sum_{n_{2}=1}^{\infty}  \frac{1}{n_{1} n_{2}} \sum_{n_3|n_1}\sum_{n_4|\frac{n_1}{n_3}}\sigma_{0,0}\left(\frac{n_1}{n_3n_4}, n_2\right) 
		S\left( \bar{a}, \pm n_{2}; \frac{q}{n_1}\right) G_{\pm} \left(\frac{n_{1}^2 n_{2}}{q^3 }\right),
	\end{align} 
	where 
	\begin{align*}
		P_1(n_1,q)=\,\frac{5}{3}\log n_1 -3\log q+3 \gamma-\frac{1}{3\tau(n_1)}\sum_{d|n_1} \log d, \hspace{1cm} \text{and}
	\end{align*} 
	\begin{align*}
		P_2(n_1,q) =\,&  (\log n_1)^2 - 5\log q\log n_1+ \frac{9}{2}(\log q)^2 + 3
		\gamma^2-3\gamma_1 + 7\gamma \log n_1 - 9 \gamma \log q\\& +\frac{1}{\tau(n_1)} \left((\log n_1+ \log q -5\gamma)\mathop{\sum}_{d|n_1}\log d-  \frac{3}{2} \mathop{\sum}_{d|n_1} (\log d)^2\right).
	\end{align*} where $\gamma$ is the Euler constant and $\gamma_1 = (\zeta(s) - (s-1)^{-1})'|_{s=1}$ be the Stieljes constant.
\end{lemma}
\begin{proof}
	See \cite{r32} for the proof. 
\end{proof}

On applying the above Lemma \eqref{lemma18} to $n$-sum with $\phi(x) = e\left(\frac{xu}{qQ}\right) V_2\left( \frac{x}{X}\right)$ and $GL(2)$ Voronoi summation formula from Lemma \eqref{voronoi} to $m,h$-sums in our main term $\mathcal{L}(H,X)$ given in equation \eqref{B21}, we arrive at,
\begin{align*}
	\mathcal{L}(H,X) = \mathcal{L}_1(H,X)+\mathcal{L}_2(H,X)+\mathcal{L}_3(H,X)+\mathcal{L}_4(H,X),
\end{align*}

where $\mathcal{L}_i(H, X), 1\leq i \leq 4$ are the expressions obtained after applying the Voronoi summation formula to $m,h,n$ sums and multiplying the four terms of $n$-sum to $m$ and $h$ sums, in the same sequence as above in Lemma \eqref{lemma18}. We notice that the analysis of $\mathcal{L}_4(H, X)$ will be exactly similar to that of $S(H, X)$ given in Lemma \ref{lemma11}. So, we will focus only on $\mathcal{L}_i(H, X), i=1,2,3$. Consider the term $\mathcal{L}_1(H,X)$ first. Let's assume the generic case $q \sim Q \asymp \sqrt{X}$ for simplicity. After applying the $GL(2)$ Voronoi summation formula, we get savings of size $H/Q$ and $X/Q$ in the sum over $h$ and $m$, respectively. Also, after applying the Voronoi summation formula to the sum over $n$, in $\mathcal{L}_1(H,X)$, it will become
\begin{align*}
	\frac{1}{2q^2}\widetilde{\phi}(1) \sum_{n_{1}|q} n_1\tau(n_1)P_2(n_1,q) S\left( \bar{a}, 0; \frac{q}{n_1}\right). 
\end{align*}
Since we can get $\widetilde{\phi}(1) \ll X$ and trivially estimating the Kloosterman sum by expanding it, we found that the $n_1$-sum is bounded by $Q$ and hence the whole above term is bounded by $X/Q$. So we get a saving of size $Q$ over the trivial bound $X$. Also, we can get a saving of size $\sqrt{Q}$ in the sum over $a$. Our total saving over the trivial bound $X^2$ becomes
\begin{align*}
	\frac{H}{Q}\,\times \frac{X}{Q}\,\times Q\,\times \sqrt{Q} = \frac{HX}{\sqrt{Q}}.
\end{align*}
This saving will be as good as we need if,
\begin{align*}
	\frac{HX}{\sqrt{Q}}\, > \,X \, \Longleftrightarrow \,\,\,\, H >  X^{1/4+\epsilon}.
\end{align*}
Hence, with this choice of size of $H$, we will get our desired result for $\mathcal{L}_1(H, X)$. In a similar way, we can treat the remaining sums $\mathcal{L}_2(H, X)$ and $\mathcal{L}_3(H, X)$ and get the same result. Hence, we will get the desired result for our main sum $\mathcal{L}(H, X)$ in Theorem \ref{TH2}.
\vspace{0.5cm}

{\bf Acknowledgement:} 
The authors wish to express their gratitude to Professor Ritabrata Munshi and the anonymous referee for their valuable comments and suggestions. They also extend their appreciation to Sumit Kumar and Prahlad Sharma for engaging in numerous insightful discussions during the work. Additionally, the authors would like to acknowledge the Indian Institute of Technology Kanpur for fostering an exceptional academic environment. It is important to note that during the course of this research, S. K. Singh received partial support from the D.S.T. Inspire Faculty Fellowship under grant number DST/INSPIRE/04/2018/000945.

{}	

\end{document}